\documentclass[12pt]{amsart}
\usepackage{geometry}
\numberwithin{equation}{section}
\usepackage{color}
\usepackage{amssymb}
\usepackage{enumerate, xspace}
\usepackage{marvosym} %monetary symbols, \EUR or \EURcr, \EyesDollar
\usepackage[sans]{dsfont}        %allows \mathds{E 1}, without sans is less heavy
\usepackage{amsthm}
\usepackage{changebar}
\usepackage[backgroundcolor=blue!20!white, linecolor=blue!20 white,textsize=footnotesize]{todonotes}
\usepackage[colorlinks]{hyperref}
\usepackage{graphicx}
\usepackage{cleveref}
\usepackage{xcolor}
\usepackage{setspace}
\usepackage{mathtools}
\usepackage{verbatim}
\usepackage{enumitem}
\makeindex
\allowdisplaybreaks[4]

\newtheorem{thm}{Theorem}[section]
\newtheorem{lem}[thm]{Lemma}

\newtheorem{defin}{Definition}[section]

\newtheorem{rem}{Remark}[section]

\calclayout

\newcommand\cE{{\mathcal E}}
\newcommand\cF{{\mathcal F}}

\newcommand\cL{{\mathcal L}}
\newcommand\cO{{\mathcal O}}
\newcommand\cM{{\mathcal M}}
\newcommand\cN{{\mathcal N}}

\newcommand{\dd}{\text{d}}
\newcommand{\x}{\vec{x}}
\newcommand{\y}{\vec{y}}

\newcommand\Ban{{\mathbb{B}}}

\newcommand\ve{\varepsilon}

\newcommand{\wh}[1]{\widehat{#1}}
\newcommand{\wt}[1]{\widetilde{#1}}
\newcommand{\ol}[1]{\overline{#1}}

\def\dist{{\rm dist}}

\def\supp{{\rm supp}}

\def\Prob{{\mathbb{P}}}

\def\EXP{{\mathbb{E}}}

\def\complex{\mathbb{C}}

\def\naturals{\mathbb{N}}

\def\Tor{\mathbb{T}}

\def\reals{\mathbb{R}}

\def\integers{\mathbb{Z}}

\def\ba{\mathbf{a}}

\def\bg{\mathbf{g}}
\def\bh{\mathbf{h}}

\def\bmu{\boldsymbol{\mu}}

\def\cF{\mathcal{F}}

\def\cE{\mathcal{E}}

\def\cL{\mathcal{L}}

\def\cM{\mathcal{M}}

\def\cN{\mathcal{N}}

\def\cO{\mathcal{O}}

\def\cQ{\mathcal{Q}}

\def\fF{\mathfrak{F}}

\def\fN{\mathfrak{N}}

\def\fn{\mathfrak{n}}

\def\tP{{\tilde P}}

\makeatother

\def\beq{\begin{equation}}
\def\eeq{\end{equation}}

\begin{document}

\title[Asymptotics for Large Deviations]{Higher order Asymptotics for  Large deviations -- Part I}
\author{Kasun Fernando and Pratima Hebbar}
\address{Pratima Hebbar\\
Department of Mathematics\\
University of Maryland \\
4176 Campus Drive\\
College Park, MD 20742-4015, United States.}
\email{{\tt phebbar@math.umd.edu}}

\address{Kasun Fernando\\
Department of Mathematics\\
University of Maryland \\
4176 Campus Drive\\
College Park, MD 20742-4015, United States.}
\email{{\tt abkf@math.umd.edu}}
%\date{ {\bf File: {\jobname}.tex.}}% eliminate in final version
\begin{abstract}
For sequences of non-lattice \textit{weakly dependent} random variables, we obtain asymptotic expansions for Large Deviation Principles. These expansions, commonly referred to as strong large deviation results, are in the spirit of Edgeworth Expansions for the Central Limit Theorem. We apply our results to show that Diophantine iid sequences, finite state Markov chains, strongly ergodic Markov chains and Birkhoff sums of smooth expanding maps \& subshifts of finite type satisfy these strong large deviation results. 
\end{abstract}
%\keywords{...}
%\subjclass{...}
\maketitle
\vspace{-20pt}
\section{Introduction}

If $\{X_n\}_{n\geq 1}$ is a sequence of independent identically distributed (iid) centred random variables with exponential moments, then Cram\'er's Large Deviation Principle (LDP) states that if $S_N=X_1+X_2+\dots+X_N$, then for all $a>0$,
\begin{align*}
\lim_{N \to \infty} \frac{1}{N}\log \Prob\big(S_N \geq Na\big) = -I(a)
\end{align*}
where $I(z)=\sup_{\theta \in \reals} \big[z\theta - \log \EXP(e^{\theta X_1})\big].$ This implies that tail probabilities of sums of iid random variables decay exponentially fast i.e.\hspace{3pt}$\Prob\big(S_N \geq Na\big) \approx e^{-I(a)N}$ for large $N$.

%In the non-iid setting, equivalent results exist for classes of random variables that are \textit{weakly dependent}. 
The following LDP (see \cite[Chapter V.6]{HH}) provides the log large deviation asymptotics for more general sequences of random variables. 

\begin{thm}[G\"artner--Ellis]\label{LocalGEThm}
Let $X_n$ be a sequence of random variables. Suppose there exists $\delta>0$ such that for $\theta \in (0,\delta)$, 
\begin{equation}\label{AsympLogMoment}
\lim_{N \to \infty} \frac{1}{N} \log \EXP(e^{\theta S_N}) = \Omega(\theta),
\end{equation}
where $\Omega$ is strictly convex continuously differentiable function with $\Omega'(0)=0$. Then, for all $a \in \left(0,\frac{\Omega(\delta)}{\delta}\right)$, there exists $\theta_a \in (0,\delta)$ such that
\begin{equation}\label{ExpRate}
\lim_{N \to \infty}\frac{1}{N} \log \Prob(S_N - N\bar{X}\geq Na)=-I(a),
\end{equation}
where $I(a)=\sup_{\theta\in (0,\delta)}[a\theta-\Omega(\theta)]=a\theta_a-\Omega(\theta_a)$ and $\bar{X}=\lim_{N \to \infty} \frac{\EXP(S_N)}{N}$.
\vspace{-8pt}
\end{thm}

It is natural to ask if the tail probabilities, $\Prob\big(S_N \geq Na\big) \approx e^{-I(a)N}$, could be made more precise. The standard approach to address this would be to look at the pre-exponential factor as well as the asymptotic expansions of the distribution function in the domain of large deviations.

\begin{defin}[Strong Asymptotic Expansions for LDP]\label{LDPExp}
Suppose $S_N$ satisfies the LDP with rate function $I$. Then, $S_N$ admits strong asymptotic expansion of order $r$ for LDP in the range $(0,L)$ if there are functions $D_k:(0,L)\to \reals$ for $0 \leq k < \frac{r}{2}$ such that for each $a \in (0,L)$,
$$\Prob(S_N-N\bar{X}\geq aN)e^{I(a)N} = \sum_{k=0}^{\lfloor r/2 \rfloor} \frac{D_k(a)}{N^{k+1/2}}+ C_{r,a}\cdot o\left(\frac{1}{N^{\frac{r+1}{2}}}\right).$$
\end{defin}

This idea of expressing the errors in limit theorems as asymptotic expansions goes back to Chebyshev in \cite{Ch}. In the setting of the Central Limit Theorem (CLT) these expansions, called the Edgeworth expansions, were first discussed rigorously in \cite{Cr}, later in \cite{ES, GK, NG1, NG2, Feller2, IL, BR, GH}, and more recently in \cite{HH, HP, KM, DF, FL}. Such expansions in the Local Limit Theorem (LLT) for iid lattice valued random variables are discussed in \cite{ES, IL}. In \cite{FL, PN}, the same expansions are considered for weakly dependent lattice random variables. 

In the absence of strong asymptotic expansions, weak expansions can be used to describe the asymptotics of large deviations. They are in the spirit of weak Edgeworth expansions in \cite{Br, FL}. 

\begin{defin}[Weak Asymptotic Expansions for LDP]\label{LDPWeakExp}
Suppose $S_N$ satisfies an LDP with rate function $I$. Let $(\cF, \|\cdot \|)$ be a normed space of functions defined on $\reals$.  Then $S_N$ admits weak asymptotic expansion of order $r$ for large deviations in the range $(0,L)$ for $f \in \cF$ if there are functions $D^f_{k}:(0,L)\to \reals$ $($depending on $f)$ for $0 \leq k < \frac{r}{2}$ such that for each $a \in (0,L)$,
$$\EXP(f(S_N-(\bar{X}+a)N))e^{I(a)N} = \sum_{k=0}^{\lfloor r/2 \rfloor} \frac{D^f_k(a)}{N^{k+1/2}}+ C_{r,a}\|f\|\cdot o\left(\frac{1}{N^{\frac{r+1}{2}}}\right).$$
\end{defin}

The asymptotic expansions for LDP have wide range of applications. One such example is the problem of obtaining uniform asymptotics for the solutions of second order parabolic equations with periodic coefficients. In the $1-$dimensional case these asymptotics have been obtained in \cite{TT}.  In the higher dimensions, the main term of the asymptotic expansion valid up to the domain of large deviations was established by the second author and her coauthors in \cite{HKN}. Such results are the key to studying the phenomena of intermittency in branching diffusion processes. 

In statisitics, inference models can be improved using these precise large deviation results. See \cite{BaR, CS, Jo} and references therein. Moreover, one can use these expansions to describe tails of invariant measures of stochastic regression models. \cite{GL, Kes} discuss similar examples. In dynamical systems, the LDP for Birkohff sums is closely related to the problem of finding rates of escape from neighbourhoods of invariant sets. This is described in \cite{Y}. Finding exact large deviations gives a better idea of the capacity of invariant sets as a barrier to transport.

Our focus here is to establish natural conditions (in the context of dynamical systems \& Markov processes) that guarantee the existence of strong and weak asymptotic expansions for LDP, and to verify them for a wide range of examples. Therefore, we do not pursue the applications mentioned before. Some of them will be a subject of a future work.

The first rigorous treatment of exact large deviation asymptotics for sums of iid random variables was done by Cram\'er in \cite{Cr} assuming the existence of an absolutely continuous component in the distribution of $X_1$. In \cite{BaR}, the strong asymptotic expansions of all orders are obtained when $X_1$ satisfies the $0-$Diophantine condition, $\lim_{|s| \to \infty} |\EXP(e^{isX_1}) |<1,$
or when $X_1$ is lattice valued. \cite{CS, Jo} describe the pre--exponential factor in large deviation asymptotics in the non--iid settng under a decay condition on the Fourier--Laplace transform of $S_N$ but do not discuss the higher order corrections. For geometrically ergodic Markov chains, these exact conndtioned are verified in \cite{KM}.

There is a substantial body of work related to the large deviation asymptotics of densities (whose existence we do not assume). See for example, \cite{IL, Pe, SS, DS}. In addition, limit theorems and their higher order asymptotics have been studied for random matrix products. For example, LLN, CLT and LDP for random matrix products can be found in \cite{LP} and \cite[Chapter 5]{BL}. In \cite{BM}, the pre-exponential factor in the LDP is obtained. It should be noted that the general criterion that we have developed here does apply in the setting studied in \cite{BM}. In fact, their results follow from \Cref{FirstTerm}. For more recent work on random matrix products see \cite{BQ, GLP, Pham}. 

Even though there is an extensive literature devoted to asymptotic expansions for the LDP, the conditions provided in them for the existence of expansions are far from being optimal. In contrast, the results presented here are sharp. In the classical setting of iid random variables, the conditions required in our paper are much less restrictive than in any of the previous work. Also, we discuss several examples in which the asymptotic expansions up to a finite order $r$ exist while expansions of order $r+1$ do not. 

We obtain weak expansions in several cases where strong expansions do not exist. In fact, we believe that our work is the first where the weak expansions are used in the context of large deviations. This is significant because the availability of weak expansions will be crucial in several applications mentioned before. Further, we obtain asymptotic expansions in several examples which were inaccessible by previous methods.

The abstract conditions guaranteeing the asymptotic expansions and the main results of the paper are presented in \Cref{Non-IID LDP}. Continuous time analogues of these results will be discussed in a sequel to this paper. The conditions we state are an extension of the Nagaev--Guivarc\textsc{\char13}h criterion, which is often used to establish the CLT for Markov processes and dynamical systems (see \cite{BL, G} for details). The idea behind the Nagaev--Guivarc\textsc{\char13}h approach is to first code the characteristic function of $S_N$ using iterations of an operator -- a Markov operator or a Perron-Forbenius transfer operator -- and then use the spectral properties of this operator to obtain results about $S_N$.

We present the proofs of our results in \Cref{Proofs}. There are two key ideas behind these proofs:  the Cram\'er's transform, which \textit{exponentially tilts} the distribution function of $S_N$ and the weak Edgeworth expansions for weakly dependent random variables found in \cite{FL}. 

In \Cref{iidCramer}, we consider the iid setting and recover the results in \cite{BaR} for non-lattice random variables. In \Cref{l-DiophVar}, we provide an af{f}irmative answer to a question raised in \cite{BaR} about the existence of strong asymptotic expansions for LDPs for iid sequences that are neither $0-$Diophantine nor lattice--valued. In \Cref{FiniteMarkov}, we show that for finite state Markov chains, weak expansions of all orders exist even when strong expansions of sufficiently high order (depending on the number of states) fail to exist. We discuss Markov chains with $C^1$--densities in \Cref{MarkovDensity}. We also discuss strong large deviation results for ergodic averages of smooth expanding maps and subshifts of finite type. These are obtained in \Cref{SFTs} and \Cref{ExpandingMaps} as a Corollary of \Cref{FirstTerm}. 

The coef{f}icients of these asymtptotic expansions are related to the asymptotic moments of the exponentially tilted $S_N$, and hence to exponential moments of $S_N$. This relationship is explicit because the coefficients are written as integrals of polynomials with coefficients depending on the exponential moments of $S_N$. The derivation of polynomials follows a standard argument due to Cram\'er, and in the non--iid setting these polynomials are described in \cite[Section 4]{FL} in detail. In fact, a precise description of the coefficients in both weak and strong asymptotic expansions along with an inductive algorithm to compute them are provided there. 

Throughout the paper, we assume that $N$ is large enough without explicitly mentioning that we do so, and make no attempt to find optimal constants in the error terms. However, we keep track of how the errors depend on the function in the weak expansion. The letter $C$ is often used to denote constants and may refer to dif{f}erent constants, even in the same sentence. The subscripts present in these constants, like $r$ and $a$ in $C_{r,a}$, describe how the constants depend on parameters. 

\section{Main Results}\label{Non-IID LDP}
Suppose that there exist a Banach space $\Ban$, a family of bounded linear operators $\cL_z:\Ban\to\Ban$, and vectors $v\in \Ban, \ell \in \Ban'$ (the space of bounded linear functionals on $\Ban$)  such that
\begin{equation}
\label{MainAssum}
\EXP\left(e^{z S_N}\right)=\ell(\cL^N_z v),
\end{equation}
for $z\in \complex$ for which the following conditions $[B]$ and $[C]$ are satisf\/ied: \vspace{5pt} \\ 
\underline{\textbf{Condition $[B]$}}: There exists $\delta > 0$ such that %\vspace{-8pt}
\begin{itemize}\setlength\itemsep{4pt}
	\item[(B1)] $z \mapsto \cL_z$ is continuous on the strip $|$Re$(z)|<\delta$ and holomorphic on the disc $|z| < \delta$.
	\item[(B2)] For each $\theta \in (-\delta, \delta)$, the operator $\cL_\theta$ has an isolated and simple eigenvalue  $\lambda(\theta) > 0$ and the rest of its spectrum is contained\ inside the disk of radius smaller than $\lambda(\theta)$ (spectral gap). In addition, $\lambda(0) = 1$.
	\item[(B3)] For each $\theta \in (-\delta, \delta)$, for all real numbers $s\neq 0$, 
	the spectrum of the operator $\cL_{\theta+is}$, denoted by sp$(\cL_{\theta+is})$, satis{f}ies: $
	\text{sp}(\cL_{\theta+is})\subseteq \{z\in \complex\ |\ |z|<\lambda(\theta)\}.$
	\item[(B4)] For each $\theta \in (-\delta, \delta)$, there exist positive numbers $r_1, r_2, K$, and $N_0$ such that $$\left\Vert \cL_{\theta+is}^N \right\Vert \leq \frac{\lambda(\theta)^N}{N^{r_2}}$$
	 for all $N>N_0$, for all $K \leq |s| \leq N^{r_1} $.%\vspace{-11pt}
	 \end{itemize}
\begin{rem}
In the case of ergodic sums of dynamical systems, $\cL_0$ is the Ruelle-Perron-Forbenius transfer operator. Also, the relation \eqref{MainAssum} takes the form $\EXP_{\mu}(e^{zS_N})=\mu(\cL^N_z \bf{1})$ where $\cL_z$ is a twisted transfer operator, $\mu$ is the initial distribution and $\bf{1}$ is the constant function $1$. In the case of Markov chains, $\cL_0$ is the corresponding Markov operator and $\EXP_{\mu}(e^{zS_N})=\mu(\cL_z^N \bf{1})$ where $\cL_z$ is a Fourier kernel associated to $\cL_0$ and $\mu$, $\bf{1}$ are as before. 
\end{rem}
\begin{rem}\label{B1-B4}
Suppose $(B4)$ holds. Let $N_1>N_0$ be such that $N^{(r_1-\epsilon)/r_1}_1>N_0$. Then, writing $N_2=N-\lceil N^{(r_1-\epsilon)/r_1} \rceil \lceil N^{\epsilon/r_1}_1 \rceil$, for all $N\gg N_1$, we have that $N_2 > N_0$ and
\begin{align*}
\phantom{aaaaaaaaaa}\frac{\|\cL^N_{\theta + is}\|}{\lambda(\theta)^N}&\leq \frac{\|(\cL^{\lceil N^{(r_1-\epsilon)/r_1} \rceil }_{\theta + is})^{ \lceil N^{\epsilon/r_1}_1 \rceil }\|}{\lambda(\theta)^{\lceil N^{(r_1-\epsilon)/r_1} \rceil \lceil N^{\epsilon/r_1}_1 \rceil }} \frac{\|\cL^{N_2}_{\theta + is}\|}{\lambda(\theta)^{N_2}}\\ &\leq \frac{\|(\cL^{\lceil N^{(r_1-\epsilon)/r_1} \rceil }_{\theta + is})\|^{\lceil N^{\epsilon/r_1}_1 \rceil }}{\lambda(\theta)^{\lceil N^{(r_1-\epsilon)/r_1} \rceil \lceil N^{\epsilon/r_1}_1 \rceil}}\leq \frac{1}{\lceil N^{(r_1-\epsilon)/r_1}  \rceil^{r_2 \lceil N^{\epsilon/r_1}_1 \rceil }},\ \ K \leq |s| \leq N^{r_1-\epsilon}.
\end{align*}
Therefore,
\begin{align*}
\|\cL^N_{\theta + is}\|\leq \frac{\lambda(\theta)^N}{N^{r_2C_{N_1}}},
\end{align*}
where $C_{N_1}=\frac{r_1-\epsilon}{r_1}N_1^{\epsilon/r_1}$. 
Note that by f{i}xing $N_1$ large enough, we can make $r_2C_{N_1}$ as large as we want. Hence, given $(B4)$, by reducing $r_1$ by an arbitrarily small quantity and choosing $N_0$ suf{f}iciently large, we may assume $r_2$ is suf{f}iciently large.
\end{rem}	 
	 
As a consequence of (B2), the operator $\cL_\theta$, $\theta \in (-\delta, \delta)$, takes the form 
\begin{equation}\label{EigenDeco}
\cL_{\theta} = \lambda(\theta)\Pi_{\theta}+ \Lambda_{\theta},
\end{equation}
where $\Pi_{\theta}$ is the eigenprojection corresponding to the top eigenvalue $\lambda(\theta)$ and  $\Pi_{\theta}\Lambda_{\theta} =\Lambda_{\theta}\Pi_{\theta} =0$.  Due to (B1), we can use perturbation theory of bounded linear operators (see \cite[Chapter 7]{Kato}) to conclude that $\theta \mapsto \lambda(\theta)$, $\theta \mapsto \Pi_{\theta}$ and $\theta \mapsto \Lambda_{\theta}$ are analytic.
\vspace{10pt}

\noindent
\underline{\textbf{Condition $[C]$}}:
For all $\theta \in (-\delta, \delta)$, $(\log \lambda)^{\prime\prime}(\theta)>0$  and  $\ell(\Pi_{\theta}v) >0$.
\vspace{5pt}
\begin{rem}\label{LDPHolds}\
\begin{enumerate}[leftmargin=1.2em]
\item[$1.$] Without loss of generality, we assume that $\bar{X}=0$ i.e.\hspace{3pt}$\{X_n\}_{n \geq 1}$ is centred to simplify the notation. One can easily reformulate the results for non--centered $\{X_n\}_{n \geq 1}$ using the corresponding results for $\{X_n- \bar{X}\}_{n \geq 1}$.
\item[$2.$] Fix $\theta \in (-\delta, \delta)$. Due to \eqref{MainAssum} and \eqref{EigenDeco} we have that\begin{align*}
\phantom{aaa}\EXP\left(e^{\theta S_N}\right) = \ell (\cL^N_\theta v) = \lambda(\theta)^N\ell\big(\Pi_{\theta}v\big) + \ell\big(\Lambda^N_{\theta}v \big)= \lambda(\theta)^N\Big[\ell\big(\Pi_{\theta}v\big) +\lambda(\theta)^{-N} \ell\big(\Lambda_{\theta}^N v \big)\Big].
\end{align*}	 	
Due to $(B2)$ and \eqref{EigenDeco}, the spectral radius of $\Lambda_\theta$ is less than $\lambda(\theta)$. So, $\lim\limits_{N \to \infty}\lambda(\theta)^{-N} \ell\big(\Lambda_{\theta}^N v \big)$ $=0$. From the condition $[C]$, $\ell\big(\Pi_{\theta}v\big)>0$.  Thus, for large enough $N,$
\[
0 < c_1 < \Big[\ell\big(\Pi_{\theta}v\big) +\lambda(\theta)^{-N} \ell\big(\Lambda_{\theta}^N v \big)\Big]< c_2
\]
for some $c_1, c_2$. Therefore 
	\[
	\lim_{N \to \infty} \frac{1}{N} \log 	\EXP\left(e^{\theta S_N}\right) = \log \lambda(\theta). 
	\]
Also, note that $\log \lambda(\theta)$ is analytic and strictly convex because $\lambda(\theta)>0$, $\lambda(\cdot)$ is analytic and $(\log \lambda)^{\prime\prime}(\theta)>0$. Also, $(\log \lambda)^\prime (0) = \frac{\lambda^\prime(0)}{\lambda(0)}=\lim_{N \to \infty} \frac{\EXP(S_N)}{N}=0$ $($see \cite[Section 4]{FL}$)$.  
%	In fact, from the strict convexity of $\log(\lambda(\theta))$ and the fact that $\log\lambda(0) = 0$, it follows that  $\log \lambda(\theta) >0 $ for all $\theta \in (-\delta, \delta), \theta\neq 0$.
Now$,$ applying \Cref{LocalGEThm}, we conclude that $S_N$ satis{f}ies the LDP in \eqref{ExpRate} with $I(z)=\sup_{\theta \in (0,\delta)}[z\theta - \log \lambda(\theta)]$.
\item[$3.$] From the above calculations it is clear that $\log \lambda(\theta) > \log(\lambda (0))=0$ for $\theta \in (0, \delta)$, and hence, $\lambda(\theta)>1$  for $\theta \in (0,\delta) $.
\item[$4.$] If $\delta=\infty$, then $B:=\lim_{\delta \to \infty}\frac{\log \lambda(\delta)}{\delta} \in (0,\infty]$ exists and the LDP holds for all $a \in (0,B)$. This is because the function $f(x)$ defined as  $f(x) = \frac{\log \lambda(x)}{x}$ is strictly increasing on $(0,\delta)$. 
In fact, the function $f$ is differentiable on $(0,\delta)$ and 
\[
f'(x) = \frac{x(\log \lambda)'(x) - (\log \lambda)(x)}{x^2}.
\]
Now, $(\log \lambda)'(x) > \frac{\log \lambda(x)}{x}$ for all $x \in (0,\infty)$ since $\log \lambda(x)$ is strictly convex. Thus, $f'(x)> 0$ for all $x \in (0,\delta)$. 
\end{enumerate}	
\end{rem}
In order to state our main results, we introduce the function space $\fF_k^m$ given by$$\fF_k^m = \{f \in C^m(\reals) |\ C^m_k(f) < \infty\},$$ where  $C^m_k(f) = \max_{0\leq j \leq m} \|f^{(j)}\|_{\text{L}^1}+  \max_{0\leq j \leq k} \|x^jf\|_{\text{L}^1}$. We call a function $f$ (left) exponential of order $\alpha$, if $\lim_{x\to -\infty} |e^{-\alpha x} f(x)| = 0$.
Define the function space $\fF^m_{k,\alpha}$ by 
$$\fF^m_{k,\alpha} = \{f \in \fF_k^m|\ f^{(m)}\text{ is exponential of order}\ \alpha\}.$$
%the collection of all $f \in \fF^{k}_m$ with $f^{(k)}$ is exponential of order $\alpha$. 
It is clear that $\fF^m_{k,\alpha} \subset \fF^m_{k, \beta}$ if $ \alpha>\beta$.  Finally, define, $\fF^m_{k,\infty} = \bigcap_{\alpha>0} \fF^m_{k, \alpha}$. 
%This intersection is non-empty. For example, the family of Gaussian functions and $C^\infty_c(\reals)$ are in $\fF^k_{m,\alpha}$ for all $\alpha>0$.
%We further assume that the sequence $\{X_n\}_{n \geq 1}$ is weakly dependent as described below. 

The following two theorems give higher order asymptotics for the LDP in \Cref{LocalGEThm} in the weak and the strong sense, respectively.	
\begin{thm}\label{WeakExp} Let $r \in \naturals$. Suppose that conditions $[B]$ and $[C]$ hold. Then, for all $a\in \Big(0, \frac{\log{\lambda(\delta)}}{\delta}\Big)$, there exist $\theta_a \in (0, \delta)$  and polynomials $P^a_{k}(x)$ of degree at most $2k$, such that for $q > \frac{r+1}{2r_1}+1$ and $\alpha > \theta_a$, for all $f \in \fF^{q}_{r+1, \alpha}$
$$\EXP(f(S_N-aN))e^{I(a)N} = \sum_{k = 0}^{\lfloor r/2 \rfloor} \frac{1}{N^{k+1/2}}\int  P^a_{k}(x) f_{\theta_a}(x)\, dx+ C^q_{r+1}(f_{\theta_a})\cdot o_{r,a}\left(\frac{1}{N^{\frac{r+1}{2}}}\right)\,\,\,\,{\text as}\,\,N \to \infty,$$
where $f_{\theta}(x)=\frac{1}{2\pi} e^{-\theta x}f(x)$ and $I(a)=\sup_{\theta\in (0,\delta)}[a\theta-\log \lambda(\theta)]=a\theta_a-\log \lambda(\theta_a).$
\end{thm}
\begin{rem}\label{UniqWeak2}
We note that for a given $a$, the polynomials $P^a_k$'s are unique. To see this, f{i}x $a$. 
From \Cref{UniqWeak}, $D^f_k(a)=\int P^a_k(x)f_{\theta_a}(x)\, dx$ are unique for all $k$. Assume there exist two polynomials, $P^a_k$ and $\tP^a_k$ with $\int P^a_k(x)f_{\theta_a}(x)\, dx = \int \tP^a_k(x)f_{\theta_a}(x)\, dx$. Since $C^\infty_c([0,1]) \subset \fF^{q}_{r+1, \alpha}$ and $\{f_{\theta_a} | f \in C^\infty_c([0,1])\}$ is dense in $L^1[0,1]$, we have for all $f \in L^1[0,1]$, $\int P^a_k(x)f(x)\, dx = \int \tP^a_k(x)f(x)\, dx$ we have that $P^a_k(x)=\tP^a_k(x)$ for $x \in [0,1]$. So, $P^a_k=\tP^a_k$.
\end{rem}
\begin{thm}\label{StrongExp}
Let $r \in \naturals$, $r \geq 2$. Suppose that conditions $[B]$ and $[C]$ hold with $r_1 > r/2$.  Then, for all $a\in \Big(0, \frac{\log{\lambda(\delta)}}{\delta}\Big)$, %there exist constants $D_k(a)$ such that 
$$\Prob(S_N \geq aN)e^{I(a)N} = \sum_{k=0}^{\lfloor r/2 \rfloor} \frac{D_k(a)}{N^{k+1/2}}+  o_{r,a}\left(\frac{1}{N^{\frac{r+1}{2}}}\right)\,\,\,\,{\text as}\,\,N \to \infty,$$
where $D_k(a)=\frac{1}{2\pi}\int_0^\infty e^{-\theta_a x}P^a_{k}(x)\, dx.$
\end{thm}

Moreover, we can evaluate the pre-exponential factor in LDPs under significantly weaker conditions. Namely, we obtain an improved version of Theorem E of \cite{HH} with  precise asymptotics. %under similar assumptions. Namely, under assumptions similar to those in Theorem E of \cite{HH}, we obtain the following theorem.
\begin{thm}\label{FirstTerm}
Suppose that $(B1), (B2), (B3)$ and $[C]$ hold. Then, for every $a \in \Big(0,\frac{\log{\lambda(\delta)}}{\delta}\Big)$, 
\[\Prob(S_N \geq aN)e^{I(a)N} =  \frac{\ell(\Pi_{\theta_a} v)\sqrt{I''(a)}}{\theta_a\sqrt{2\pi N}}\Big(1 + o(1)\Big)\,\,\,\,{\text as}\,\,N \to \infty.\] 
\end{thm}
\begin{rem}
Analogous results hold for $a \in \Big(\frac{\log(\lambda(-\delta))}{-\delta},0 \Big)$. In fact, one can deduce the corresponding results for $a<0$ by considering $\{-X_n\}_{n \geq 1}$ and functions that are right exponential of order $\alpha$. However, for simplicity we focus only on $a>0$.  
\end{rem}

\section{Proofs of the main results}\label{Proofs}

Recall from \Cref{LDPHolds} that the LDP given by \eqref{ExpRate} holds under the conditions $[B]$ and $[C]$. That is, given $a \in \Big(0, \frac{\log{\lambda(\delta)}}{\delta}\Big)$, there exists $\theta_a \in (0,\delta)$ such that $$\lim_{N \to \infty}\frac{1}{N} \log \Prob(S_N \geq aN)=-I(a),$$
where $I(a)=\sup_{\theta\in (0,\delta)}[a\theta-\log \lambda(\theta)]=a\theta_a-\log \lambda(\theta_a).$ So we fix $a \in \Big(0, \frac{\log{\lambda(\delta)}}{\delta}\Big)$, and take $\theta_a$ to denote value of $\theta \in (0,\delta)$ for which $I(a)$ is achieved. Since $\theta_a$ is the unique maximizer of analytic function $f(\theta)=a\theta - \log \lambda (\theta)$ on $(0,\delta)$, $f'(\theta_a)=0$. That is,
\begin{equation}\label{Derivative}
a=\frac{\lambda'(\theta_a)}{\lambda(\theta_a)}
\end{equation}

	 \begin{proof}[Proof of Theorem \ref{WeakExp}]
	 	Observe that
	 	\begin{align*}
	 	\EXP(f(S_n-an))e^{a\theta_a n} &= \EXP(e^{\theta_a S_n}e^{-(S_n-an)\theta_a}f(S_n-an)) \\ &= \int \widehat{f}_{\theta_a}(s) e^{-iasn} \ell(\cL^n_{\theta_a+is}v) \, ds,
	 	\end{align*}
	 	where $f_{\theta_a}(x)=\frac{1}{2\pi} e^{-\theta_a x}f(x)$.
	 	Define, $\overline{\cL}_{s}=\frac{e^{-ias}}{\lambda(\theta_a)}\cL_{\theta_a+is}$. %where $\lambda(\theta)$ is the top eigenvalue of $\cL_\theta$. 
	 	Then, 
	 	\begin{align*}
	 	\EXP(f(S_n-an))e^{a\theta_a n} &=\lambda(\theta_a)^n\int \widehat{f}_{\theta_a}(s) \ell(\overline{\cL}^n_{s} v) \, ds.
	 	\end{align*}
	 	From this, we have
	 	\begin{equation}\label{BaseLDP}
	 	\EXP(f(S_n-an))e^{I(a)n} = \EXP(f(S_n-an))e^{[a\theta_a -\log \lambda(\theta_a)]n} = \int \widehat{f}_{\theta_a}(s) \ell(\overline{\cL}^n_{s}v) \, ds.
	 	\end{equation}
	 	%In particular, 
	 	%\begin{equation}
	 	%= \int \widehat{f}_{\theta_a}(s) \ell(\overline{\cL}^N_{s}v) \, ds.
	 	%\end{equation}
	 	
% Combining this with $$\EXP(e^{\theta S_N}) = \lambda(\theta)^N\ell(\Pi_\theta v)+\ell(\Lambda_\theta^N v)$$ we conclude that, \begin{align*} \ell(\Pi_\theta v)= \lim_{N \to \infty} \frac{\EXP(e^{\theta S_N})}{\lambda(\theta)^N}\end{align*} for all $\theta$. 	 
The following lemma (whose proof we postpone till the end of the proof of the theorem) allows us to obtain the asymptotics of \eqref{BaseLDP}. 
\begin{lem}\label{WeakEXPforOp} 
Suppose conditions $[B]$ and $[C]$ hold. Let $r \in \naturals$. Then, for all $a \in \Big(0, \frac{\log{\lambda(\delta)}}{\delta}\Big),$ there are polynomials $P^a_{k}(x)$ of degree at most $2k$, such that for $g\in \fF^{q}_{r+1}$,  $q > \frac{r+1}{2r_1} + 1$,
$$\int \wh{g}(s) \ell(\overline{\cL}^N_{s}v) \, ds=\sum_{k=0}^{\lfloor r/2 \rfloor} \frac{1}{N^{k+1/2}}\int  P^a_{k}(x) g(x)\, dx+ C^q_{r+1}(g) \cdot o_{r,a}\left(\frac{1}{N^{\frac{r+1}{2}}}\right).$$
\end{lem}
\noindent
We refer to this expansion as the weak expansion of $\ol{\cL}_s$ for $g \in \fF^q_{r+1}$.

Since $f \in \fF^{q}_{r+1, \alpha}$ with $\alpha>\theta_a$, we have that $f_{\theta_a} \in \fF^{q}_{r+1}$. We show this when $r=0$ and $q=1$. The argument for general $q$ and $r$ is similar. Suppose, $f(x),f'(x),xf(x) \in L^1$, $f'(x)$ is continuous and exponential order $\alpha > \theta_a$. It is clear that $(e^{-\theta_a x}f(x))^\prime = -\theta_a e^{-\theta_a x}f(x) + e^{-\theta_a x}f'(x)$ is continuous. We need to show that $e^{-\theta_a x}f(x)$, $(e^{-\theta_a x}f(x))^\prime$ and $x e^{-\theta_a x}f(x)$ are absolutely integrable. Since $f^\prime$ is  exponential of order $\alpha$, given $\epsilon >0$, there exists an $M > 0$ such that for all $x \leq -M$,
$-\epsilon e^{\alpha x} \leq f'(x) \leq \epsilon e^{\alpha x}$,
and therefore, 
\[
-\int_{-\infty}^{x}\epsilon e^{\alpha y}\, dy \leq \int_{-\infty}^{x}f'(y)\, dy \leq \int_{-\infty}^{x}\epsilon e^{\alpha y}\, dy.
\]
In addition, our assumptions imply that $\lim_{|x| \to \infty}f(x) = 0$. Thus,
$
-\frac{\epsilon}{\alpha} e^{\alpha x} \leq f(x) \leq \frac{\epsilon}{\alpha}  e^{\alpha x},
$
which shows that $f$ is also exponential of order $\alpha$.

Now, it remains to show that $e^{-\theta_a x}f(x), e^{-\theta_a x}f^{\prime}(x), x e^{-\theta_a x}f(x) \in L^1$. This is true because there is $M>0$ such that for $x<-M$, $|e^{-\theta_a x}f^\prime(x)|<e^{(\alpha - \theta_a)x} $, $|e^{-\theta_a x}f(x)| < e^{(\alpha - \theta_a)x} $ and $|xe^{-\theta_a x}f(x)| < -x e^{(\alpha - \theta_a)x} $. 

Finally, to complete the proof of \Cref{WeakExp} we apply \Cref{WeakEXPforOp} to $f_{\theta_a}$. 
\end{proof}

\begin{proof}[Proof of Lemma \ref{WeakEXPforOp}]
For a fixed $\theta_a \in (-\delta, \delta)$, from \eqref{EigenDeco} and perturbation theory of bounded linear operators (see \cite[Chapter 7]{Kato}), there exists $\delta_1  \in (0, \delta)$ such that for all $|s| \leq \delta_1$, $\cL_{\theta_a + is}$ can be expressed as  
\begin{equation}\label{OpDecom}
 \cL_{\theta_a + is} =\lambda(\theta_a+is)\Pi_{\theta_a+is}+\Lambda_{\theta_a+is},
\end{equation}
where $\Pi_{\theta_a+is}$ is the eigenprojection to the top eigenspace of $\cL_{\theta_a+is}$, the spectral radius of $\Lambda_{\theta_a+is}$ is less than $|\lambda(\theta_a+is)|$, and $\Lambda_{\theta_a+is}\Pi_{\theta_a+is} = \Pi_{\theta_a+is}\Lambda_{\theta_a+is} = 0$. In addition, the spectral data are analytic with respect to the perturbation parameter because the perturbations are analytic. That is, $z\mapsto \lambda(z)$, $z \mapsto \Pi_{z}$ and $z \mapsto \Lambda_z$ are analytic in a neighbourhood of $z_0=\theta_a + i0$ (see \cite[Chapter 7]{Kato}).  

Iterating \eqref{OpDecom}, we obtain 
\begin{equation}\label{ItrOpDecom}
	\cL^n_{{\theta_a+is}}=\lambda(\theta_a+is)^n\Pi_{\theta_a+is}+\Lambda^n_{\theta_a+is}.
	\end{equation}
Define $\ol{\Pi}_s = \Pi_{\theta_a+is}$ and $\ol{\Lambda}_s = \frac{e^{-ias}}{\lambda(\theta_a)}\Lambda_{\theta_a+is}$. Then, for all $|s| < \delta_1$,
\begin{equation}\label{Normalized}
\ol\cL^n_{s} =\ol\mu(s)^n\ol\Pi_{s}+\ol\Lambda^n_{s},
\end{equation}
where $\ol\cL_{s}=\frac{e^{-ias}}{\lambda(\theta_a)}\cL_{\theta_a+is}$ and $\ol\mu(s)=\frac{e^{-ias}\lambda(\theta_a+is)}{\lambda(\theta_a)} $.

%Note that the top eigenvalue of $\overline{\cL}_{s}$ is $\overline{\mu}(s)=\frac{e^{-ias}}{\lambda(\theta_a)}\lambda(\theta_a+is)$.
From \eqref{Derivative} and the condition $[C]$,
\begin{equation}\label{derivatives}
\begin{split}
\overline{\mu}(0)=1,\ \ \overline{\mu}^\prime&(0)=\frac{d}{ds}\overline{\mu}(s)\Big|_{s=0}= -ia + i
\frac{\lambda'(\theta_a)}{\lambda(\theta_a)}=0\ \ \text{and}\\
\overline{\mu}''(0)=&a^2-\frac{\lambda''(\theta_a)}{\lambda(\theta_a)}=-(\log \lambda)^{\prime\prime}(\theta_a)=:-\sigma^2_a 
\end{split}
\end{equation}   
for some $\sigma_a>0$. Thus, there exists $\ol\delta$ such that 
\begin{equation}\label{EigenBound}
|\overline{\mu}(s)| \leq e^{-\sigma_a^2s^2/4},\ |s|<\ol\delta.
\end{equation}	 	

First, we estimate the contribution to $\int \wh{g}(s) \ell(\overline{\cL}^N_{s}v) \, ds$  from the region away from $s = 0$.
Fix $\ol\delta>0$ as in \eqref{EigenBound}. 
Due to (B3), the spectral radius of $\ol{\cL}_{s}$ is strictly less than $1$. Since $s \mapsto \ol{\cL}_{s}$ is continuous, there exists $c_0 \in (0,1)$ such that $\|\overline{\cL}^n_{s}\|\leq c^n _0$ for all $\delta \leq |s| \leq K$ ($K$ as in (B4)). Thus,
$$\bigg| \int_{\ol\delta<|s|<K} \wh{g}(s)  \ell(\overline{\cL}^n_{s}v )\, ds \bigg|\leq C\|g\|_1c^n_0.$$ Due to \Cref{B1-B4}, without loss of generality we assume that $r_2 > r_1+(r+1)/2$. From (B4), 
\begin{align*}
\bigg| \int_{K<|s|<n^{r_1}} \wh{g}(s)\ell(\overline{\cL}^n_{s}v )\, ds \bigg| \leq \frac{C\|g\|_1}{\lambda(\theta_a)^n}\int_{K<|s|<n^{r_1}} \|\cL^n_{\theta_a+is} \|\, ds  &\leq \frac{C\|g\|_1}{n^{r_2-r_1}} \\ &=\|g\|_1 \cdot o(n^{-(r+1)/2}).
\end{align*}

Since $g \in \fF^{q}_{r+1}$, we have that $s^q\wh{g}(s)=(-i)^q\wh{g^{(q)}}(s)$ and $\wh{g^{(q)}}$ is bounded. Therefore,
\begin{align}\label{AtInftyLDP}
\bigg| \int_{|s|>n^{r_1}} \wh{g}(s)\ell(\overline{\cL}^n_{s}v ) \, ds \bigg| \leq C\int_{|s|>n^{r_1}} |\wh{g}(s)| \, ds &\leq C\int_{|s|>n^{r_1}}  \Big|\frac{\widehat{g^{(q)}}(s)}{s^q}\Big| \, ds \\ &\leq C\frac{\|\widehat{g^{(q)}}\|_{\infty}}{n^{r_1(q- 1)}} \nonumber \\ &=  C^{q}_{r+1}(g) \cdot o (n^{-(r+1)/2}). \nonumber
\end{align}
Note that the integral $\int_{|s|>n^{r_1}}  \Big|\frac{1}{s^q}\Big| \, ds$ is finite since  $q > \frac{r+1}{2r_1} + 1 > 1$. 
Combining these estimates, we obtain 
\begin{equation}\label{AwayFrm0Err}
\bigg| \int_{|s|>\ol\delta} \wh{g}(s)\ell(\overline{\cL}^n_{s}v ) \, ds \bigg| = C^{q}_{r+1}(g)\cdot o(n^{-(r+1)/2}).
\end{equation}

From \eqref{EigenBound}, we know that for all  $|s| < \ol{\delta}\sqrt{n}$,
$|\ell(\ol{\cL}^n_{s/\sqrt{n}}v)| \leq C e^{-\frac{1}{4}\sigma_a^2 s^2}.$
Thus, for $\sqrt{D\log n} \leq |s| \leq \ol{\delta}\sqrt{n}$, 
$|\ell(\ol{\cL}^n_{s/\sqrt{n}}v)| \leq C n^{-\sigma_a^2 D}.$
Therefore, 
\begin{align*}
\bigg| \int_{\sqrt{\frac{D\log n}{n}} \leq |s| \leq \ol{\delta}} \wh{g}(s)\ell(\overline{\cL}^n_{s}v )\, ds \bigg|  &= \bigg| \int_{\sqrt{{D\log n}} \leq |u| \leq \ol{\delta}\sqrt{n}} \wh{g}(\frac{u}{\sqrt{n}})\ell(\overline{\cL}^n_{\frac{u}{\sqrt{n}}}v )\, \frac{du}{\sqrt{n}} \bigg| \\& \leq C \frac{1}{n^{\sigma_a^2 D}}\bigg| \int_{\sqrt{{D\log n}} \leq |u| \leq \ol{\delta}\sqrt{n}} \wh{g}(\frac{u}{\sqrt{n}}) \frac{du}{\sqrt{n}} \bigg|  \\& = C \frac{1}{n^{\sigma_a^2 D}}\bigg| \int_{\sqrt{\frac{D\log n}{n}} \leq |s| \leq \ol{\delta}} \wh{g}(s) \, ds \bigg| \\&\leq \frac{2C\ol{\delta}\|g\|_1}{n^{\sigma_a^2 D}}.
\end{align*}
Choosing $D > \frac{r+1}{2\sigma^2_a}$, we have 
\begin{equation}\label{MidEstLDP}
\bigg| \int_{\sqrt{\frac{D\log n}{n}} \leq |s| \leq \ol{\delta}} \wh{g}(s)\ell(\overline{\cL}^n_{s}v )\, ds \bigg|  = C^{q}_{r+1}(g) \cdot o(n^{-(r+1)/2}).
\end{equation}

Using (B3) and compactness, there exist $C>0$ and $\epsilon \in (0,1)$ (which do not depend on $n$ and $s$) such that $\|\ol\Lambda^n_s\|\leq C\epsilon^n$ for all $|s|\leq\delta_1$. By \eqref{Normalized}, 
\begin{equation}\label{eq:char fn}
\ell(\overline{\cL}^n_{s/\sqrt{n}} v)= \ol\mu\Big(\frac{s}{\sqrt{n}}\Big)^n \ell\big( \ol\Pi_{s/\sqrt{n}} v \big) + \ell\big(\ol\Lambda^n_{s/\sqrt{n}} v\big).
\end{equation}
Let us focus on the first term of \eqref{eq:char fn}. Put $Z(s)=\ell( \ol\Pi_{s} v)$. Note that $Z(s)$ is analytic on $|s|<\delta_1$ because $s \mapsto \ol{\Pi}_s$ is analytic.

Now we are in a position to compute $P^a_k(x)$. To this end we make use of ideas in \cite{FL}. From $\eqref{derivatives}$, function $\log\ol\mu$ can be written as
\begin{align*}
\log \ol\mu\Big(\frac{s}{\sqrt{n}}\Big) = -\frac{\sigma^2_as^2}{2n} + \psi\Big(\frac{s}{\sqrt{n}}\Big),
\end{align*}
where $\psi$ denotes the error term, $\psi(0)=\psi'(0)=\psi''(0)=0$ and $\psi(s)$ is analytic. That is 
\begin{align*}
\ol\mu\Big(\frac{s}{\sqrt{n}}\Big)^n=e^{-\frac{\sigma^2_as^2}{2}}\exp \Big( n\psi\Big(\frac{s}{\sqrt{n}}\Big)\Big).
\end{align*}
Denote by $s^2\psi_r(s)$ the order $(r+2)$ Taylor approximation of $\psi$. Then, $\psi_r$ is the unique polynomial such that $\psi(s)=s^2\psi_r(s)+o(|s|^{r+2})$. Also, $\psi_r(0)=0$ and $\psi_r$ is a polynomial of degree $r$. In fact, we can write $\psi(s)=s^2\psi_r(s)+s^{r+2}\tilde{\psi}_r(s)$, where $\tilde{\psi}_r$ is analytic and $\tilde{\psi}_r(0)=0$. Thus,
\begin{align*}
\exp \Big( n\psi\Big(\frac{s}{\sqrt{n}}\Big) \Big)=\exp\Big(s^2 \psi_r\Big(\frac{s}{\sqrt{n}}\Big)+\frac{1}{n^{r/2}}s^{r+2}\tilde{\psi}_r\Big(\frac{s}{\sqrt{n}}\Big)\Big).
\end{align*}

%By \eqref{eq:char fn} we have,
%\begin{equation}\label{CharFn1} \ell(\overline{\cL}^n_{s/\sqrt{n}} v)=e^{-\frac{\sigma^2_as^2}{2}} \exp \Big( n\psi\Big(\frac{s}{\sqrt{n}}\Big) \Big)Z\Big(\frac{s}{\sqrt{n}}\Big) + \ell(\ol\Lambda^n_{s/\sqrt{n}} v).
%\end{equation}
Denote by $Z_r(s)$ the order $r$ Taylor expansion of $Z(s)-Z(0)$. Then, $Z_r(0)=0$ and $Z(s)=Z(0)+Z_r(s)+s^{r}\tilde{Z}_r(s)$ with analytic $\tilde{Z}_r(s)$ such that $\tilde{Z}_r(0)=0$. Now, substituting the Taylor expansions for $\log \ol \mu (s)$ and $Z(s)$, and taking $\overline{Z}_r$ to be the remainder of $\log Z(s)$ when approximated by powers of $Z_r$ up to order $r$:
\begin{align*}
e^{\frac{\sigma^2_as^2}{2}}&\ol\mu^n\Big(\frac{s}{\sqrt{n}}\Big)Z\Big(\frac{s}{\sqrt{n}}\Big) \nonumber \\&= e^{\frac{\sigma^2_as^2}{2}}\ol\mu^n\Big(\frac{s}{\sqrt{n}}\Big)\exp \log Z\Big(\frac{s}{\sqrt{n}}\Big) \nonumber \\ &=Z(0)\exp\Big(s^2 \psi_r\Big(\frac{s}{\sqrt{n}}\Big)+\frac{1}{n^{r/2}}s^{r+2}\tilde{\psi}_r\Big(\frac{s}{\sqrt{n}}\Big) \nonumber \\ &\phantom{=\exp\Big(t^2 \psi_r\Big(\frac{t}{\sqrt{n}}\Big)++} +\sum_{k=1}^{r}\frac{(-1)^{k+1}}{kZ(0)^{k}}\Big[Z_r\Big(\frac{s}{\sqrt{n}}\Big)\Big]^k +\frac{1}{n^{r/2}Z(0)}s^r\overline{Z}_r\Big(\frac{s}{\sqrt{n}}\Big)\Big)  \nonumber \\&=Z(0) \Big[ 1+ \sum_{m=1}^r \frac{1}{m!} \Big[s^2 \psi_r\Big(\frac{s}{\sqrt{n}}\Big) + \sum_{k=1}^{r}\frac{(-1)^{k+1}}{k Z(0)^{k}}\Big(Z_r\big(\frac{s}{\sqrt{n}}\big)\Big)^k\Big]^m\Big]  \nonumber \\ \nonumber &\phantom{=\exp\Big(s} +Z(0) \Big[ \frac{1}{n^{r/2}}s^{r+2}\tilde{\psi}_r\Big(\frac{s}{\sqrt{n}}\Big) +\frac{1}{n^{r/2}Z(0)}s^r\overline{Z}_r\Big(\frac{s}{\sqrt{n}}\Big)+s^{r+1}\cO\big(n^{-\frac{r+1}{2}}\big)\Big]. \nonumber 
\end{align*}

Take $\varphi(s)= ns^2Z(0)\tilde{\psi}_r(s)+\overline{Z}_r(s)$. It is clear that $\varphi(s)$ is analytic and $\varphi(0)=0$. Now, collecting terms in the RHS according to ascending powers of $n^{-1/2}$ we obtain,
\begin{align}\label{PolyComp}
e^{\frac{\sigma^2_as^2}{2}}&\ol\mu^n\Big(\frac{s}{\sqrt{n}}\Big)Z\Big(\frac{s}{\sqrt{n}}\Big)= \sum_{k=0}^{r} \frac{A_k(s)}{n^{k/2}}+\frac{s^{r}}{n^{r/2}}\varphi\Big(\frac{s}{\sqrt{n}}\Big)+s^{r+1}\cO\big(n^{-\frac{r+1}{2}}\big). 
\end{align}

Notice that, $A_k(s)$ (as a function) and $k$ (as an integer) have the same parity. To see this, note that for each $k\geq 0$,  $A_k$'s are formed by collecting terms with the common factor of $n^{-k/2}$. Observe that $\psi_r$ and $Z_r$ are a polynomial in $\frac{s}{\sqrt{n}}$ with no constant term, and therefore when we take powers of $s^2\psi_r\Big(\frac{s}{\sqrt{n}}\Big)$ and $Z_r\Big(\frac{s}{\sqrt{n}}\Big)$, the resulting $A_k$ will contain terms of the form $c_m s^{2m+k}$. 

Note that $A_0 \equiv Z(0)$. The highest power of $s$ in $A_k$, $k \geq 1$, is a result from the term $C s^2 \frac{s}{\sqrt{n}}$ in $s^2 \psi_r\big(\frac{s}{\sqrt{n}}\big)$ being raised to its $k^{\text{th}}$ power, i.e., $m=k$ above. Thus, $A_k$ are polynomials of degree $3k$. The lowest power of $s$ in $A_k$ corresponds to $m=0$ and is equal to $k$. Next, define $\beta_{n,r}$ by
\begin{align}\label{MainPolyExp}
\beta_{n,r}(s)=\sum_{k=0}^{r} \frac{A_k(s)}{n^{k/2}}.
\end{align}

We write the Taylor approximation of $\wh{g}$:
$$\widehat{g}(s) = \sum_{j=0}^r \frac{\widehat{g}^{(j)}(0)}{j!}s^j+ \frac{s^{r+1}}{(r+1)!}\widehat{g}^{(r+1)}(\epsilon(s)),$$
where $0\leq| \epsilon(s)| \leq |s|$ and $$|\widehat{g}^{(r+1)}(\epsilon(s))|=\bigg|\int x^{r+1}e^{-i\epsilon(s)x}g(x) \, dx\bigg| \leq \int |x^{r+1}g(x)| \, dx \leq C^{0}_{r+1}(g).$$ Therefore,
\begin{align*}
\int_{|s|<\sqrt{D \log n}} &\widehat{g}\Big(\frac{s}{\sqrt{n}}\Big)\ell(\overline{\cL}^n_{s/\sqrt{n}}v ) \, ds \\ &= \sum_{j=0}^r \frac{\widehat{g}^{(j)}(0)}{j!n^{j/2}} \int_{|s|<\sqrt{D\log n}}s^j\ell(\overline{\cL}^n_{s/\sqrt{n}}v ) \, ds \nonumber  \\ &\phantom{aaaaa}+ \frac{1}{n^{(r+1)/2}}\frac{1}{(r+1)!}\int_{|s|<\sqrt{D\log n}}\ell(\overline{\cL}^n_{s/\sqrt{n}}v ) s^{r+1}  \widehat{g}^{(r+1)}\Big(\epsilon\Big(\frac{s}{\sqrt{n}}\Big)\Big)\, ds, \nonumber
\end{align*}
where
\begin{align*}
\bigg|\int_{|s|<\sqrt{D\log n}} \ell(\overline{\cL}^n_{s/\sqrt{n}}v ) s^{r+1} &\widehat{g}^{(r+1)}\Big(\epsilon\Big(\frac{s}{\sqrt{n}}\Big)\Big)\, ds \bigg|\leq C^0_{r+1}(g) \int |s|^{r+1} e^{-cs^2} \, ds
\end{align*}
for large $n$. Hence,
\begin{multline}\label{TaylorfHat}
\int_{|s|<\sqrt{D \log n}} \widehat{g}\Big(\frac{s}{\sqrt{n}}\Big) \ell(\overline{\cL}^n_{s/\sqrt{n}}v ) \, ds \\ = \sum_{j=0}^r \frac{\widehat{g}^{(j)}(0)}{j!n^{j/2}} \int_{|s|<\sqrt{D\log n}}s^j\ell(\overline{\cL}^n_{s/\sqrt{n}}v ) \, ds  + C^0_{r+1}(g)\cdot \cO(n^{-(r+1)/2}).
\end{multline}
From $\eqref{PolyComp}$, 
\begin{align}\label{FreqAsympPoly}
e^{\frac{\sigma^2_as^2}{2}}\ell(\overline{\cL}^n_{s/\sqrt{n}}v )&=\exp \Big( n\psi\Big(\frac{s}{\sqrt{n}}\Big) \Big)Z\Big(\frac{s}{\sqrt{n}}\Big) + e^{\frac{\sigma^2_as^2}{2}}\ell\big(\ol\Lambda^n_{s/\sqrt{n}}v \big) \nonumber \\  
&= \sum_{k=0}^{r} \frac{A_k(s)}{n^{k/2}}+\frac{s^{r}}{n^{r/2}}\varphi\Big(\frac{s}{\sqrt{n}}\Big)+C_{r,a} \cdot \cO\Big(\frac{\log^{(r+1)/2}(n)}{n^{(r+1)/2}}\Big)
\end{align}
for $|s|<\sqrt{D \log n}$.
Substituting this in \eqref{TaylorfHat},
\begin{align}\label{Near0WeakEdge}
&\int_{|s|<\sqrt{D \log n}} \widehat{g}\Big(\frac{s}{\sqrt{n}}\Big) \ell(\overline{\cL}^n_{s/\sqrt{n}}v )\, ds \\ &= \sum_{j=0}^r \frac{\widehat{g}^{(j)}(0)}{j!n^{j/2}} \int_{|s|<\sqrt{D\log n}}s^j e^{-\sigma^2_as^2/2}\sum_{k=0}^{r} \frac{A_k(s)}{n^{k/2}} \, ds + C^0_{r+1}(g) \cdot \cO\Big(\frac{\log^{(r+1)/2}(n)}{n^{(r+1)/2}}\Big) \nonumber \\ &= \sum_{k=0}^{r}  \sum_{j=0}^r  \frac{\widehat{g}^{(j)}(0)}{j!n^{(k+j)/2}} \int_{|s|<\sqrt{D\log n}}s^j A_k(s) e^{-\sigma^2_as^2/2} \, ds + C^0_{r+1}(g)\cdot o(n^{-r/2}). \nonumber
\end{align}

Since $A_k$ and $k$ have the same parity, if $k+j$ is odd then $$\int_{|s|<\sqrt{D\log n}}s^j A_k(s) e^{-\sigma^2_as^2/2} \, ds = 0.$$
So only the positive integer powers of $n^{-1}$ will remain in the expansion. Also, there is $C$ that depends only on $r$ and $a$ such that
\begin{align*}
\int_{|s|\geq \sqrt{D\log n}}s^j A_k(s) e^{-\sigma^2_as^2/2} \, ds \leq C \int_{|s|\geq \sqrt{D\log n}}s^{4r} e^{-\sigma^2_as^2/2} \, ds  \leq \frac{C_{r,a}}{n^{\sigma^2_aD/4}}.
\end{align*}
Choosing $D$ such that $2\sigma_a^2D>(r+1)/2$, $$\int s^j A_k(s) e^{-\sigma^2_as^2/2} \, ds= \int_{|s|\leq \sqrt{D\log n}}s^j A_k(s) e^{-\sigma^2_as^2/2} \, ds + C_{r,a}\cdot o(n^{-r/2}).$$
Therefore, fixing $D$ large, we can assume the integrals to be over the whole real line. 

Now, define $b_{kj}=\int s^j A_k(s) e^{-\sigma^2_a s^2/2} \, ds $ 
and substitute $\widehat{g}^{(j)}(0)= \int (-is)^j g(s) \, ds$
in \eqref{Near0WeakEdge} to obtain
\begin{align*}
\int_{|s|<\sqrt{D \log n}}\widehat{g}&\Big(\frac{s}{\sqrt{n}}\Big)\ell(\overline{\cL}^n_{s/\sqrt{n}}v ) \, ds \\ &= \sum_{k=0}^{r}  \sum_{j=0}^r  \frac{b_{kj}}{j!n^{(k+j)/2}} \int (-is)^j g(s) \, ds  + C^0_{r+1}(g) \cdot o(n^{-r/2}) \\ &= \sum_{m=0}^r \frac{1}{n^m} \int g(s) \sum_{k+j=2m} \frac{b_{kj}}{j!} (-is)^j  \, ds + C^0_{r+1}(g) \cdot o(n^{-r/2}) \nonumber \\&= \sum_{m=0}^{\lfloor r/2 \rfloor} \frac{1}{n^m} \int g(s) P_m^a(s) \, ds + C^0_{r+1}(g) \cdot o(n^{-r/2}), \nonumber
\end{align*}
where 
\begin{equation}\label{LDPEXPPoly}
P^a_{m}(s)= \sum_{k+j=2m} \frac{b_{kj}}{j!}(-is)^j.
\end{equation}
Combining, the above with \eqref{AtInftyLDP} and \eqref{MidEstLDP} we obtain the required result. 
\end{proof}

Take $F_N$ to be the distribution function of $S_N$. Let $\wt{S}_N$ be a function defined on some finite measure space $(\Omega, \cF, \wt P)$ such that it induces the finite measure $\frac{e^{\theta_a x}\, }{\lambda(\theta_a)^N}dF_N(x)$ on $\reals$. Note that $\wt{S}_N$ is not a random variable since the measure it induces on $\reals$ is not a probability measure. Take $G_N(x)$ to be the distribution function of $\frac{\wt{S}_N-aN}{\sqrt{N}}$. That is,
\begin{equation}\label{DistStilde}
G_N(x) = \tilde P\bigg( \frac{\wt{S}_N-aN}{\sqrt{N}} \leq x \bigg).
\end{equation} 
Then, from the definition of the operator $\bar{\cL}$ in \eqref{Normalized}, we obtain $\wh{G}_N(s \sqrt{N})=\ell(\ol{\cL}^N_{s}v)$ for all $s \in \reals$ because  $$\int e^{i\frac{x-aN}{\sqrt{N}}s} \frac{e^{\theta_a x}\, }{\lambda(\theta_a)^N}dF_N(x)=\ell (\ol{\cL}_{s/\sqrt{N}}^Nv) .$$ 
Also, recall that for $|s|<\delta_1 \sqrt{N}$, where $\delta_1$ is as in proof of \Cref{WeakEXPforOp}, 
$$\ell (\ol{\cL}_{s/\sqrt{N}}^Nv) = \ol\mu\Big(\frac{s}{\sqrt{N}}\Big)^N\ell(\ol\Pi_{s/\sqrt{N}} v)+\ell(\ol\Lambda^N_{s/\sqrt{N}} v).$$
From \eqref{PolyComp} and the estimate $\|\ol\Lambda^N_{s}\|\leq C\epsilon^N$ (which is explained below the equation \eqref{MidEstLDP}) for $|s|<\delta_1 $, we conclude that RHS converges to $Z(0)e^{-\frac{\sigma^2_as^2}{2}}$ as $N\to \infty$. Hence, $G_N(x)$ converges to the function  $Z(0)\fN(x)$, where $\fn(x)=\frac{1}{\sqrt{2\pi \sigma^2_a}}e^{-\frac{x^2}{2\sigma^2_a}}$ and $\fN(x)=\int_{-\infty}^x \fn(y) \, dy$. 

We denote the function inducing the measure  $Z(0)\fN(x)$ on the real line by $Z(0)\cN(0, \sigma_a^2)$.
Thus, $\frac{\wt{S}_N-aN}{\sqrt{N}}$ converges weakly to $Z(0)\cN(0, \sigma_a^2)$.

Observe that
\begin{align*}
\int (x-aN)&\frac{e^{\theta_a x}\, }{\lambda(\theta_a)^N}dF_N(x)\\ & = \frac{d}{ds}\Big|_{s = 0}\int e^{i(x-aN)s}\frac{e^{\theta_a x}\, }{\lambda(\theta_a)^N}dF_N(x)\\
& = \frac{d}{ds}\Big|_{s = 0}\Big(\ol\mu(s)^N\ell(\ol\Pi_s v)+\ell(\ol\Lambda^N_s v)\Big) \\&= N\ol\mu(0)^{N-1}\ol\mu'(0)\ell(\ol\Pi_0 v) + \ol\mu(0)^N\frac{d}{ds}\Big|_{s = 0} \ell(\ol\Pi_s v) +  \frac{d}{ds}\Big|_{s = 0} \ell(\ol\Lambda^N_s v).
\end{align*}
From \eqref{derivatives}, we have $\ol\mu'(0) = 0$, $\ol{\mu}(0) = 1$, and therefore 
\begin{align}\label{MeanAsym}
\frac{1}{N}\int (x-aN)\frac{e^{\theta_a x}\, }{\lambda(\theta_a)^N}dF_N(x)=\frac{1}{N}\ell(\ol\Pi'_0 v) + \frac{1}{N}\ell ((\ol\Lambda^N_0)' v).
\end{align}

We claim $\|(\ol\Lambda^N_0)'\| \leq C \epsilon^N$ for some $\ve \in (0,1)$. Since $\cL_0$ has a spectral gap, by perturbation theory, there exists $\ve \in (0,1)$ (uniform for $|s| \leq \delta_1$) such that $$\text{sp}(\cL_s) \subset \{z \in \complex||z|<\epsilon\}\cup \{\lambda(\theta+is)\}.$$
Therefore, $$\ol\Lambda^N_s = \frac{1}{2\pi i} \int_\Gamma z^N (z-\ol{\cL}_s)^{-1} \, dz$$
where $\Gamma$ is the positively oriented circle centered at $z=0$ with radius $\ve$. Hence, 
\begin{align*}
\ol\Lambda^N_s - \ol\Lambda^N_s = \frac{1}{2\pi i} \int_\Gamma z^N (z-\ol{\cL}_o)^{-1}(\ol\cL_s-\ol\cL_0)(z-\ol{\cL}_s)^{-1} \, dz.
\end{align*}
Now the claim follows from the observation that $\ol\cL_s-\ol\cL_0 = \cO(|s|)$. So, both the terms on RHS of \eqref{MeanAsym} converge to zero as $N \to \infty$. Therefore, $\wt{S}_N$ has asymptotic mean $a$. 

We say that $\ol{\cL}_{s}$ admits a strong asymptotic expansion of order $r$ if $\wt{S}_N$ admits the Edgeworth expansion of order $r$, i.e., there exist polynomials $Q_k$ (whose parity as a function is the opposite of the parity of $k$) such that 
\begin{equation}\label{EdgeFiniteMeas}
G_N(x)-Z(0)\fN(x)=Z(0)\sum_{k=1}^r \frac{Q_k(x)}{N^{k/2}}\fn(x) + o(N^{-r/2})
\end{equation}
uniformly for $x \in \reals$, where $\fn(x)=\frac{1}{\sqrt{2\pi \sigma^2_a}}e^{-\frac{x^2}{2\sigma^2_a}}$ and $\fN(x)=\int_{-\infty}^x \fn(y) \, dy$.  Note that these expansions, if they exist, are unique (the argument in \Cref{UniqStrng} applies).
	 
The proof of the existence of the strong expansions is based on two intermediate lemmas (\Cref{Strong->Weak} and \Cref{LimitAsymp} below). The first lemma establishes that whenever the order $r$ strong asymptotic expansion for $\ol{\cL}_{s}$ exists, lower order weak expansions (as in \Cref{WeakEXPforOp}) exist for $g\in \fF^1_r$. It is the Proposition A.1 in \cite{FL} adapted to our setting. The second lemma shows that whenever $\ol{\cL}_{s}$ has weak expansions for $g \in \fF^1_{r}$ the corresponding $S_N$ has strong expansions (of the corresponding order) for large deviations. Finally, to prove \Cref{StrongExp}, we have to show that the conditions $[B]$ and $[C]$ imply the existence of strong expansions for $\ol{\cL}_{s}$.

\begin{lem}\label{Strong->Weak}
Suppose that $\ol{\cL}_{s}$ admits the order $r$ strong asymptotic expansion. Then there are polynomials $P_k$ such that
\begin{align*}
\int \wh{g}(s)\ell(\ol{\cL}^n_{s})\, ds =\sum_{m=0}^{\lfloor (r-1)/2\rfloor}\frac{1}{n^{m+\frac{1}{2}}}\int_\reals P_{m}(s)g(s)\, ds+ C^1_{r}(g) \cdot o\left(n^{-r/2}\right).
\end{align*}
for $g \in \fF^1_r$. %$($We refer to this expansion as the order $r-1$ weak expansion of $\ol{\cL}_{s}$ for $g \in \fF^1_r$.$)$  
\end{lem}
\begin{proof}
Suppose $g \in \fF^1_{r}$. Define, 
$\cE_{r,n}(x) = Z(0)\fN(x)+Z(0)\sum_{k=1}^r \frac{Q_k(x)}{n^{k/2}} \fn(x)$.
Observe that $G_n(x)-\cE_{r,n}(x)=o(n^{-r/2})$ uniformly in $x$ and
\begin{align*}
d\cE_{r,n}(x) &= Z(0)\fn(x)\, dx + Z(0)\sum_{k=1}^r \frac{1}{n^{k/2}} \left[Q'_k\left(x\right)\fn\left(x\right) + Q_k(x)\fn'(x)\right]\, dx \\ &=Z(0)\sum_{p=0}^r \frac{R_k(x)}{n^{p/2}} \fn(x)\, dx,
\end{align*}
where $R_k$ are polynomials given by $R_k=Q'_k+Q_kQ$ and $Q$ is such that $\fn'(x)=Q(x) \fn(x)$, i.e., 
\begin{equation}\label{FourierPoly1}
\fn(x)R_k(x) = \frac{d}{dx}\Big[\fn(x)Q_k(x)\Big].
\end{equation}
Note that $R_k$ and $Q_k$ are of opposite parity, because $Q(x)$ is of degree $1$. 

Next, we observe that 
\begin{align*}
\int \wh{g}(s)\ell(\ol{\cL}^n_{s})\, ds &= \int \wh{g}(s)\wh{G}_n(s \sqrt{n})\, ds \\ &= \frac{1}{\sqrt{n}}\int \wh{g}\Big( \frac{s}{\sqrt{n}} \Big) \wh{G}_n(s)\, ds \\ &=\int g(x\sqrt{n}) \, dG_n(x)\ \ \ (\text{by Plancherel}) \\ &= \int g(x\sqrt{n}) \, d\cE_{r,n}(x) +\int g(x\sqrt{n})\, d(G_n -\cE_{r,n})(x).
\end{align*}

Now, we integrate by parts and use that $g(\pm \infty)=0$ (because $g \in \fF^1_r$) and the fact that $\cE_{r,n}(\pm\infty)$, $G_n(\pm\infty)$ are finite to obtain
\begin{align}\label{WGEXP}
\int \wh{g}(s)\ell(\ol{\cL}^n_{s})\, ds &=  \int g(x\sqrt{n}) \, d\cE_{r,n}(x) + (G_n-\cE_{r,n})(x)g(x\sqrt{n}) \Big|_{-\infty}^{\infty} \nonumber \\ &\ \hspace{165pt}  - \int (G_n  -\cE_{r,n})(x)\sqrt{n} g^\prime(x\sqrt{n})\, dx \nonumber \\ &= \int \sum_{k=0}^r \frac{1}{n^{k/2}}R_k(x)\fn(x)\, g(x\sqrt{n})dx+ o\left(n^{-r/2}\right)\int\sqrt{n} g^\prime(x\sqrt{n})\, dx  \nonumber \\ &=\sum_{k=0}^r \frac{1}{n^{k/2}} \int R_k(x)\fn(x)\, g(x\sqrt{n})dx+ \|g^\prime\|_1 \cdot o\left(n^{-r/2}\right).
\end{align}

From the Plancherel formula, 
$$\int \sqrt{n}g\big(x\sqrt{n}\big)R_k(x) \fn(x) \, dx = \frac{1}{2\pi}\int \wh{g}\Big(\frac{s}{\sqrt{n}}\Big) A_k(s) e^{-\frac{\sigma^2_as^2}{2}} \, ds, $$
where $\wh{R_k \fn} (s)= A_k(s)e^{-\frac{\sigma^2_as^2}{2}}$ and $A_k$ are given by the following relation,
\begin{equation}\label{FourierPoly}
A_k(s)e^{-\frac{\sigma_a^2 s^2}{2}} = R_k\left(-i\frac{d}{ds}\right)\Big[e^{-\frac{\sigma_a^2 s^2}{2}} \Big].
\end{equation}
%\begin{equation}\label{FourierPoly}
%R_k(s)e^{-\frac{s^2}{2\sigma_a^2}} =\frac{1}{\sqrt{2\pi\sigma_a^2}}A_k\left(-i\frac{d}{ds}\right)\Big[e^{-\frac{s^2}{2\sigma_a^2}} \Big]
%\end{equation}
This follows from the basic Fourier identity $\wh{x^j f}(s) = (-i)^j\frac{d^j}{ds^j}\wh{f}(s)$, and we refer the reader to \cite[Chapter III, IV]{ES} for a detailed discussion. We also note that, by the uniqueness of expansions, these $A_k$ agree with the ones in \eqref{FreqAsympPoly}.
Also, by construction, $R_k$ and $A_k$ have the same parity. This means $A_k$ has the same parity as $k$.

Next, replace $\int R_k(x)\fn(x)\, g(x\sqrt{n})\, dx$ by $\frac{1}{2\pi \sqrt{n}}\int\wh{g}\big(\frac{s}{\sqrt{n}}\big) A_k(s) e^{-\frac{\sigma^2_as^2}{2}} \, ds$ 
in \eqref{WGEXP} to obtain  
\begin{align*}
\int \wh{g}(s)\ell(\ol{\cL}^n_{s})\, ds =\frac{1}{2\pi}\sum_{k=0}^r \frac{1}{n^{(k+1)/2}} \int\wh{g}\Big(\frac{s}{\sqrt{n}}\Big)A_k(s) e^{-\frac{\sigma^2_as^2}{2}} \, ds +  \|g^\prime\|_1 \cdot o\left(n^{-r/2}\right).
\end{align*}
Then, substituting $\wh{g}$ with its order $r-1$ Taylor expansion, 
\begin{align*}
\int \wh{g}(s)\ell(\ol{\cL}^n_{s})\, ds  =\frac{1}{2\pi}\sum_{k=0}^r\sum_{j=0}^{r-1} \frac{\wh{g}^{(j)}(0)}{j!n^{(j+k+1)/2}} \int s^j e^{-\sigma^2_a s^2/2} A_k(s)\, ds + C^1_{r}(g) \cdot o\left(n^{-r/2}\right).
\end{align*}
Put
$$b_{jk}=\frac{1}{2\pi}\int s^j e^{-\sigma^2_a s^2/2} A_k(s)\, ds\ \ \text{and}\ \ \wh{g}^{(j)}(0)=\int (-is)^jg(s) \, ds$$
to obtain
\begin{align*}
\int \wh{g}(s)\ell(\ol{\cL}^N_{s})\, ds =\sum_{k=0}^r\sum_{j=0}^{r-1} \frac{b_{jk}}{j!n^{(j+k)/2}} \int (-is)^j g(s)\, ds + C^1_{r}(g) \cdot o\left(n^{-r/2}\right).
\end{align*}

Since $k$ and $A_k$ are of the same parity, $b_{jk}=0$ when $j+k$ is odd. So we collect terms such that $j+k=2m$ where $m=0,\dots,r-1$ and write
$$P_m(s)=\sum_{j+k=2m}\frac{b_{jk}}{j!}(-is)^j.$$
Then rearranging, simplifying and absorbing higher order terms into the error, we obtain 
\begin{align*}
\int \wh{g}(s)\ell(\ol{\cL}^n_{s})\, ds =\sum_{m=0}^{\lfloor (r-1)/2\rfloor}\frac{1}{n^{m+\frac{1}{2}}}\int P_{m}(s)g(s)\, ds+ C^1_{r}(g) \cdot o\left(n^{-r/2}\right).
\end{align*}
This is the order $r-1$ weak expansion for $g \in \fF^1_r$. 
\end{proof}

\vspace{-10pt}
\begin{lem}\label{LimitAsymp}
Suppose $\{f_k\}$ is a sequence in $\fF^1_{r+1}$ satisfying the following:
\begin{enumerate}
\item[$(a)$] There exists $C>0$ such that $C^1_{r+1}(f_k) \leq C$ for all $k$,
\item[$(b)$] $f_k$ are uniformly bounded in $L^\infty(\reals)$,
\item[$(c)$] $f_k \to f$ pointwise,  
\item[$(d)$] For all $m$,
\begin{equation*}\label{IntCon}
\lim_{k\to \infty}\int  P_{m}(s) f_k(s) \, ds =  \int  P_{m}(s) f(s) \, ds,
\end{equation*}
\item[$(e)$] There exists $N_0$ such that for all $N>N_0$,
\begin{align*}
\EXP(f_k(S_N-aN))e^{I(a)N} =\sum_{m=0}^{\lfloor r/2\rfloor}\frac{1}{N^{m+\frac{1}{2}}}\int_\reals P_{m}(s){f}_k(s)\, ds+ C^1_{r}({f}_k) \cdot o\left(N^{-(r+1)/2}\right).
\end{align*}
\end{enumerate}
Then, for $N>N_0$,
$$ \EXP(f(S_N-aN))e^{I(a)N} = \sum_{m=0}^{\lfloor r/2 \rfloor} \frac{1}{N^{m+\frac{1}{2}}}\int  P_{m}(s) f(s)\, ds + C \cdot o(N^{-(r+1)/2}).$$
\end{lem}
\begin{proof}
From $(e)$ and $(a)$ for $N>N_0$,
\begin{align}\label{Assum1}
\Big|\EXP(f_k(S_n-an))e^{I(a)n} -\sum_{m=0}^{\lfloor r/2 \rfloor} \frac{1}{n^{m+\frac{1}{2}}}\int  P_{m}(s) f_k(s)\, ds \Big| &\leq C^1_{r+1}(f_k)\cdot o(n^{-r/2}) \\&\leq C \cdot o(n^{-r/2}).  \nonumber
\end{align}
Now, $(b)$ and $(c)$ give us that 
$$\lim_{k \to \infty} \EXP(f_k(S_n-an))=\EXP(f(S_n-an)).$$
This along with assumption $(d)$ allow us to take the limit $k \to \infty$ in the RHS of \eqref{Assum1} and to conclude that
$$\Big|\EXP(f(S_n-an))e^{I(a)n}-\sum_{m=0}^{\lfloor r/2 \rfloor} \frac{1}{n^{m+\frac{1}{2}}}\int  P_{m}(s) f(s)\, ds \Big| \leq C \cdot o(N^{-r/2}).$$
This implies the result.
\end{proof}

\begin{proof}[Proof of \Cref{StrongExp}]
Let $a \in \Big(0, \frac{\log{\lambda(\delta)}}{\delta}\Big)$.
From \eqref{DistStilde}, note that $$G_n(\infty)=\int \frac{e^{\theta_a x}}{\lambda(\theta_a)^n}\, dF_n(x)=\frac{\EXP(e^{\theta_a S_n})}{\lambda(\theta_a)^n}= Z(0)+\frac{\ell(\Lambda^n_{\theta_a} v)}{\lambda(\theta_a)^n}$$ 
and
\begin{align*}
\wh{G}_n(s)=\frac{e^{-\frac{isan}{\sqrt{n}}}\ell(\cL^n_{\theta_a+is/\sqrt{n}} v)}{\lambda(\theta_a)^n}= \ell(\overline{\cL}^n_{s/\sqrt{n}} v).
\end{align*}
%where $\overline{\cL}_{s}=\frac{e^{-isa}\cL_{\theta_a+is}}{\lambda(\theta_a)}$.
We proceed as in \Cref{WeakEXPforOp} (see \eqref{OpDecom}--\eqref{MainPolyExp}) and obtain the polynomials $A_k$ and $\beta_{r,n}$. 
%\begin{align*}
%\beta_{r+1,n}(s)=\sum_{k=0}^{r+1} \frac{A_k(s)}{n^{k/2}}
%\end{align*}
%with $A_0(s)=\ell(\ol{\Pi}_0 v)=Z(0)$. 
Also, define polynomials $R_k$ and $Q_k$ using the relations \eqref{FourierPoly} and \eqref{FourierPoly1}, respectively. Then define
%$$R_k(s)e^{-\frac{s^2}{2\sigma_a^2}} =\frac{1}{\sqrt{2\pi\sigma_a^2}}A_k\left(-i\frac{d}{ds}\right)\Big[e^{-\frac{s^2}{2\sigma_a^2}} \Big]$$ 
%and $Q_k$ by $$ \fn(x)R_k(x) = \frac{d}{dx}\Big[\fn(x)Q_k(x)\Big].$$
$$\beta_{r+1,n}(x) = Z(0)\fN(x)+Z(0)\sum_{k=1}^{r+1} \frac{Q_k(x)}{n^{k/2}} \fn(x) \, \, \text{and}\
 \ol{\beta}_{r+1,n}(s)=e^{-\frac{\sigma_a^2s^2}{2}}\frac{\beta_{r+1,n}(s)}{Z(0)}.$$
Then, $Z(0)\ol{\beta}_{r+1,n}(s)$ is the Fourier transform of $\cE_{r+1,n}(x)$. This follows from the definitions of these quantities.

% and for a proof see \cite[Chapter III, IV]{ES}.

\iffalse{{ \color{red} recheck how many Z(0)'s you need}}\fi
 
From the Berry-Ess\'een inequality, \cite[Lemma 12.2]{BR}, for each $\ve>0$ there exists $B>0$ such that
\begin{align}\label{BEIneq}
\Big|G_n(x)-&\big(1+\lambda(\theta_a)^{-n}Z(0)^{-1}\ell(\Lambda^n_{\theta_a} v)\big)\cE_{r+1,n}(x)\Big| \nonumber \\ &\leq \frac{1}{\pi}\int_{-Bn^{\frac{r+1}{2}}}^{Bn^{\frac{r+1}{2}}} \bigg|\frac{\ell({\overline{\cL}^n_{s/\sqrt{n}} v}) - \big(Z(0)+\lambda(\theta_a)^{-n}\ell(\Lambda^n_{\theta_a} v)\big)\ol{\beta}_{r+1,n}(s)}{s} \bigg|\, ds + \frac{\ve}{n^{\frac{r+1}{2}}}.
\end{align}

Note that $\big(\ell(\Lambda^n_{\theta_a+is/\sqrt{n}} v)-\ell(\Lambda^n_{\theta_a} v)\ol{\beta}_{r+1,n}(s)\big)\big|_{s=0}=0$ because $\ol{\beta}_{r+1,n}(0)=1$. Also, both $\ell\big(\ol\Lambda^n_{s/\sqrt{n}} v\big)$ and $\ol{\beta}_{r+1,n}(s)$ are uniformly bounded in $s$ and $n$. Therefore, choosing $\gamma < \delta_1\ (\delta_1$ as in \eqref{OpDecom}$)$, we have 
\begin{align}\label{EssError}
\frac{1}{\pi}\int_{-\gamma \sqrt{n}}^{\gamma \sqrt{n}} &\bigg|\frac{\ell\big(\ol\Lambda^n_{s/\sqrt{n}} v\big) - \lambda(\theta_a)^{-n}\ell(\Lambda^n_{\theta_a} v)\ol{\beta}_{r+1,n}(s)}{s} \bigg|\, ds \\ &= \frac{\lambda(\theta_a)^{-n}}{\pi}\int_{-\gamma\sqrt{n}}^{\gamma \sqrt{n}} \bigg|\frac{\ell\big(\Lambda^n_{\theta_a+is/\sqrt{n}} v\big) -\ell(\Lambda^n_{\theta_a} v)\ol{\beta}_{r+1,n}(s)}{s} \bigg|\, ds \nonumber  \\ &\leq C\lambda(\theta_a)^{-n}\sqrt{n} = o(n^{-\frac{r+1}{2}}). \nonumber 
\end{align}

We claim that 
\begin{equation}\label{Near0Strng}
\frac{1}{\pi}\int_{-\gamma \sqrt{n}}^{\gamma \sqrt{n}}  \bigg|\frac{\ol{\mu}(s/\sqrt{n})^n Z(s/\sqrt{n}) - {\beta}_{r+1,n}(s)}{s} \bigg|\, ds = o(n^{-\frac{r+1}{2}})
\end{equation}
for sufficiently small $\gamma$. From the definition of $\beta_{r+1,n}(s)$,
\begin{align*}
\frac{\ol{\mu}(s/\sqrt{n})^n Z(s/\sqrt{n})- e^{-\frac{\sigma_a^2s^2}{2}}\beta_{r+1,n}(s)}{s} = \frac{e^{-\frac{\sigma^2_as^2}{2}}}{n^{(r+1)/2}}\Big(s^{r}\varphi\Big(\frac{s}{\sqrt{n}}\Big)+s^{r+1}\cO\big(n^{-\frac{r+2}{2}}\big) \Big),
\end{align*}
where $\varphi(s)=o(1)$ as $s\to 0$. As a result, for all $\ve>0$ the integrand of \eqref{Near0Strng} can be made smaller than $\frac{\ve}{n^{(r+1)/2}}(s^{r}+s^{r+1})e^{-\frac{\sigma^2_as^2}{2}}$  by choosing $\gamma$ small enough. This establishes \eqref{Near0Strng}.  

Combining \eqref{EssError} and \eqref{Near0Strng}, we obtain that for small $\gamma$,
\begin{equation}\label{Near0Err}
\frac{1}{\pi}\int_{-\gamma \sqrt{n}}^{\gamma \sqrt{n}} \bigg|\frac{\ell({\overline{\cL}^n_{s/\sqrt{n}} v}) - \big(Z(0)+\lambda(\theta_a)^{-n}\ell(\Lambda^n_{\theta_a} v)\big)\ol{\beta}_{r+1,n}(s)}{s} \bigg|\, ds \leq \frac{C \ve}{n^{(r+1)/2}},
\end{equation}
where $C=\int (s^{r}+s^{r+1})e^{-\frac{\sigma^2_as^2}{2}} \, ds$.

%and assumptions (B3) and (B4). To see this we split the LHS to the following three integrals.
Take
\begin{align*}
J_1 &= \frac 1\pi\int_{\gamma\sqrt{n}<|s|<K\sqrt{n}} \bigg|\frac{\ell({\overline{\cL}^n_{s/\sqrt{n}} v})}{s} \bigg|\, ds,\\
J_2 &= \frac 1\pi\int_{K\sqrt{n}<|s|<Bn^{\frac{r+1}{2}}} \bigg|\frac{\ell({\overline{\cL}^n_{s/\sqrt{n}} v})}{s} \bigg|\, ds, \\
J_3 &= \frac 1\pi\int_{\gamma\sqrt{n}<|s|<Bn^{\frac{r+1}{2}}} e^{-\frac{\sigma_a^2s^2}{2}}\bigg|\frac{{\beta}_{r+1,n}(s)}{s} \bigg|\, ds,
\end{align*} 
where $K$ is as in (B4).

Now we estimate the these integrals using (B3) and (B4). Since $\beta_{r+1,n}(s)$ is a polynomial of $\cO(1)$ as $n \to \infty$, $e^{-\frac{\sigma_a^2s^2}{4}}{\beta}_{r+1,n}(s)$ %has the form $\beta_{r+1,n}(s)=e^{-\frac{\sigma^2_a s^2}{2}}R_n(s)$, where $R_n(s)$ is a polynomial and $e^{-\frac{\sigma^2_a s^2}{4}}R_n(s)$ 
is bounded uniformly in $s$ and $n$ (say by $M$). Therefore, 
$$J_3 \leq M\int_{|s|>\delta\sqrt{n}}e^{-\frac{\sigma^2_a s^2}{4}} \, ds \leq Me^{-cn}$$
for some $c>0$. By (B4), $ \|\ol{\cL}^n_{s}\| \leq \frac{1}{n^{r_2}}$ with $r_2>r+1$ (WLOG) for $K<|s|<n^{r_1}$. Also, by assumption, $r_1>r/2$.  Thus, 
$$J_2 = \frac 1\pi\int_{K<|s|<Bn^{r/2}} \bigg|\frac{\ell({\overline{\cL}^n_{s} v})}{s} \bigg|\, ds\leq Cn^{r/2-r_2} = o(n^{-\frac{r+1}{2}}).$$
By (B3), the spectral radius of $\ol{\cL}_{s}$ is strictly less than $1$. Since $s \mapsto \ol{\cL}_{s}$ is continuous, there exist $\gamma<1$ and $C>0$ such that $ \|\ol{\cL}^n_s\|\leq C\gamma^n $ for all $\delta \leq |s| \leq K$ for large $n$. Then, for sufficiently large $n$, we have 
\begin{align*} 
J_1=\frac 1\pi\int_{\delta<|s|<K} \bigg|\frac{\ell({\overline{\cL}^n_{s} v})}{s} \bigg|\, ds \leq C\gamma^n.
\end{align*}
Combining the asymptoics for $J_1, J_2$ and $J_3$,
\begin{equation}\label{Awayfrom0}
\frac{1}{\pi}\int_{|s|>\gamma \sqrt{n}} \bigg|\frac{\ell({\overline{\cL}^n_{s/\sqrt{n}} v}) - \big(Z(0)+\lambda(\theta_a)^{-n}\ell(\Lambda^n_{\theta_a} v)\big)\ol{\beta}_{r+1,n}(s)}{s} \bigg|\, ds = o (n^{-(r+1)/2}).
\end{equation} 

From \eqref{Near0Err} and \eqref{Awayfrom0},  we deduce that RHS of \eqref{BEIneq} is $o(n^{-\frac{r+1}{2}})$. Therefore,
$G_n(x)=\big(1+\lambda(\theta_a)^{-n}Z(0)^{-1}\ell(\Lambda_{\theta_a} v)\big)\cE_{r+1,n}(x)+o(n^{-\frac{r+1}{2}})$ uniformly in $x$. Since $\cE_{r+1,n}(x)$ is uniformly bounded in $x, n$ and $\lambda(\theta_a)>1$ we have that $\frac{\ell(\Lambda_{\theta_a} v)\cE_{r+1,n}(x)}{\lambda(\theta_a)^nZ(0)}$ decays exponentially fast. Thus, 
$G_n(x)=\cE_{r+1,n}(x)+o(n^{-\frac{r+1}{2}})$. By the derivation of $\cE_{r+1,n}(x)$, it is immediate that this expansion takes the form described in \eqref{EdgeFiniteMeas}. 

From \Cref{Strong->Weak}, $\ol\cL_{s}$ has the order $r$ weak expansion on $\fF^1_r$. Since $f \in \fF^1_{r,\alpha}$ where $\alpha>\theta_a$, we have that $f_{\theta_a} \in \fF^1_{r}$. Therefore,
\begin{align*}
\EXP(f(S_n-an))e^{I(a)n} &= \int \widehat{f}_{\theta_a}(s) \ell(\overline{\cL}^n_{s}v) \, ds\\ &=\frac{1}{2\pi}\sum_{k=0}^{\lfloor r/2 \rfloor} \frac{1}{n^{k+\frac{1}{2}}}\int e^{-\theta_a z} P^a_{k}(z) f(z) dz+ C^1_{r+1}(g) \cdot o_{r,\theta_a}\left(n^{-\frac{r+1}{2}}\right).
\end{align*}
for all $f \in \fF^1_{r,\alpha}$, $\alpha>\theta_a$.

In particular, this holds for $f \in C^\infty_c(\reals)$. Let $\{f_m \} \subset C^\infty_c(\reals)$ be a sequence such that $1_{[0,\infty)}$ is a point-wise limit of $f_m$ and $(f_m)_{\theta_a}$'s satisfy the hypothesis of \Cref{LimitAsymp}. (We construct such a sequence in \Cref{fk}). Then, by \Cref{LimitAsymp}, 
\begin{align*}
\EXP(1_{[0,\infty)}(S_n-an))e^{I(a)n} &=\frac{1}{2\pi}\sum_{k=0}^{\lfloor r/2 \rfloor} \frac{1}{n^{k+\frac{1}{2}}}\int e^{-\theta_a x} P^a_{k}(x) 1_{[0,\infty)}(x)\, dx+  o_{r,\theta_a}\left(n^{-\frac{r+1}{2}}\right).
\end{align*}
That is 
\begin{align*}
\Prob(S_n\geq an)e^{I(a)n} &=\frac{1}{2\pi}\sum_{k=0}^{\lfloor r/2 \rfloor} \frac{1}{n^{k+\frac{1}{2}}}\int_0^\infty e^{-\theta_a x} P^a_{k}(x)\, dx+  o_{r,\theta_a}\left(n^{-\frac{r+1}{2}}\right).
\end{align*}
\end{proof} 

\begin{rem}
Note that the coefficients of the strong expansion are obtained by replacing $f$ with $1_{[0,\infty)}$ in coefficients of the weak expansions. Since $f_k$'s are bounded in $\fF^1_{r+1}$, we can do this without altering the order of the error. However, for any $q>1$, $1_{[0,\infty)}$ is not a pointwise limit of a sequence of functions $f_k$ in $\fF^{q}_r$ with $C^q_{r+1}(f_k)$ bounded. To observe this, assume that $\|f_k\|_1, \|f^{\prime}_k\|_1, \|f^{\prime \prime}_k\|_1$ are uniformly bounded and $f_k \to 1_{[0,\infty)}$ point-wise. Then, for all $\phi \in C^\infty_c(\reals)$,
\begin{align*}\int \delta'\ \phi = - \int \delta\ \phi^\prime = \int 1_{[0,\infty)}\ \phi^{\prime \prime} = \lim_{k \to \infty} \int f_k\ \phi^{\prime \prime} = \lim_{k \to \infty} - \int f^\prime_k\ \phi^{\prime} = \lim_{k \to \infty} \int f_k^{\prime \prime}\ \phi.
\end{align*} 
This implies that $\frac{|\phi'(0)|}{ \|\phi\|_\infty}\leq\sup_{k} \|f_k^{\prime \prime}\|_1$ for all $\phi \in C^\infty_c(\reals)$. Clearly, this is a contradiction. Therefore, \Cref{WeakExp} does not automatically give us strong expansions. Indeed, in Section \ref{Eg} we exhibit an example (see example \ref{DiscFinAtm}) where weak expansions exist when strong expansions fail to exist. 
\end{rem}

The proof of \Cref{FirstTerm} is similar to that of \Cref{StrongExp}. We include it for completeness. 

\begin{proof}[Proof of \Cref{FirstTerm}] 
Let $a \in \Big(0,\frac{\log{\lambda(\delta)}}{\delta}\Big)$.
Since (B1) and (B2) hold, as before we have \eqref{PolyComp}, where $\varphi$ is analytic, $\varphi(0)=0$ and $r=1$. As in the previous proof, Berry-Ess\'een inequality, \cite[Lemma 12.2]{BR}, given $\ve>0$, there exists $B>0$ such that
\begin{align*}
|G_{n}(x)-&\big(1+\lambda(\theta_a)^{-n}Z(0)^{-1}\ell(\Lambda_{\theta_a} v)\big)\cE_{1,n}(x)| \\ &\leq \frac 1\pi\int_{-B\sqrt{n}}^{B\sqrt{n}}\bigg|\frac{\ell({\overline{\cL}_{s/\sqrt{n}} v}) - \big(Z(0)+\lambda(\theta_a)^{-n}\ell(\Lambda_{\theta_a} v)\big)\ol{\beta}_{1,n}(s)}{s} \bigg|\, ds+\frac{\ve}{\sqrt{n}} .
\end{align*} 
Since $\varphi(t)=o(1)$ as $t \to 0$, we have
\begin{align*}
\frac{\ol{\mu}(s/\sqrt{n})^n Z(s/\sqrt{n})- e^{-\frac{\sigma^2_a s^2}{2}}\beta_{1,n}(s)}{s} = \frac{e^{-\frac{\sigma^2_as^2}{2}}}{\sqrt{n}}\Big(\varphi\Big(\frac{s}{\sqrt{n}}\Big)+s\cO\big(n^{-1}\big) \Big).
\end{align*}
Also, we conclude that
$$\frac 1\pi\int_{\gamma\sqrt{n}<|s|<B\sqrt{n}} e^{-\frac{\sigma^2_a s^2}{2}}\bigg|\frac{{\beta}_{r+1,n}(s)}{s} \bigg|\, ds = \cO(e^{-cn})$$
as before. 
Because of (B3), there is $\gamma<1$ such that
\begin{align*}
\int_{\delta\sqrt{n}<|s|<B\sqrt{n}}\left|\frac{\ell({\overline{\cL}_{s/\sqrt{n}} v}) }{s}\right|\, ds = \int_{\delta<|s|<B}\left|\frac{\ell({\overline{\cL}_{s} v}) }{s}\right|\, dt \leq C \sup_{\gamma\leq |s|\leq B}\|\ol\cL^n_s\| \leq C\gamma^n.
\end{align*}
Combining these estimates, we conclude that $\ol\cL^n_s$ admits the strong expansion of order $1$. Therefore, $\ol\cL^n_s$ admits the weak expansion order $0$ for $f \in \fF^1_1$. As before, approximating $1_{[0,\infty)}$ by a sequence in $C^\infty_c$, we conclude that
\begin{align*}
\Prob(S_n\geq an)e^{I(a)n} &=\frac{1}{\sqrt{n}}\frac{1}{2\pi}\int_0^\infty e^{-\theta_a x} P^a_{0}(x)\, dx+  o_{r,\theta_a}\left(\frac{1}{\sqrt{n}}\right).
\end{align*}
From \eqref{LDPEXPPoly}, $P^a_0(x)=Z(0)\sqrt{\frac{2\pi}{\sigma^2_a}}=\ell(\Pi_{\theta_a} v)\sqrt{\frac{2\pi}{\sigma^2_a}}$. Then,
\begin{align*}
\frac{1}{2\pi}\int_0^\infty e^{-\theta_a z} P^a_{0}(z) dz = \frac{\ell(\Pi_{\theta_a} v)}{\sqrt{2\pi \sigma^2_a} } \int_{0}^\infty e^{-\theta_a z} dz = \frac{\ell(\Pi_{\theta_a} v)}{\sqrt{2\pi \sigma^2_a} } \frac{1}{\theta_a}.
\end{align*}
From the duality of the Legendre transform, $\sigma^2_a=(\log \lambda)''(\theta_a)=\frac{1}{I''(a)}$ . Hence, we have the required form of the first order expansion. 
\end{proof}

\begin{rem}
$(B1)$ through $(B4)$ with $r_1>r/2$ imply that $\overline{\cL}_{s}$ satisfies the conditions $(A1)$ through $(A4)$ in \cite{FL} with $r_1>r/2$. We observed above that this is enough to guarantee the existence of the order $r+1$ Edgeworth expansion for $\wt{S}_N$. However, we cannot directly apply the results in \cite{FL} because $\wt{S}_N$ does not induce a probability measure.
\end{rem}

\section{Examples}\label{Eg}
\subsection{iid\hspace{3pt}random variables with Cram\'er's condition}\label{iidCramer}

Let $X$ be a non-lattice centred random variable whose logarithmic moment generating function $h(\theta)=\log \EXP(e^{\theta X})$ is finite in a neighborhood of $0$, denoted by $J$. Let $X_n$ be a sequence of iid copies of $X$. Then, from \cite[Chapter 1]{Ho}, we have the LDP:
$$
\lim_{N \to \infty} \frac{1}{N}\log \Prob(S_N \geq Na ) = - I(a),\ \text{if}\ a>0,
$$
%and
%$$
%\lim_{N \to \infty} \frac{1}{N}\log \Prob(S_N \leq Na ) = - I(a),\ \text{if}\ a<0
%$$
where the rate function $I$ is given by $$I(z)=\sup_{\gamma \in \reals} \big[\gamma z - \log \EXP(e^{\gamma X}) \big].$$ For each $a \in (0, \rm{essup}(X))$, there exists a unique $\theta_a$ such that $I(a)=\theta_a z - \log \EXP(e^{\theta_a X})$. 
 
We further assume that $X$ satisfies the Cram\'er's condition. That is,
\begin{equation}\label{Cramer}
\limsup_{|t|\to \infty} |\EXP(e^{itX})| < 1.\vspace{-7pt}
\end{equation} 
This is equivalent to $X$ being $0-$Diophantine, a notion we define later in \eqref{l-Dioph}. These conditions are enough to guarantee the existence of weak and strong expansions for large deviations:

\begin{thm}
Let $X$ be a non--lattice centered random variable whose logarithmic moment generating function is finite in a neighborhood of $0$, and which satisfies the Cram\'er's condition. Let $X_n$ be a sequence of iid copies of $X$. Then, for all $r$, 
\begin{enumerate}
\item[$(a)$] $S_N$ admits the weak asymptotic expansion of order $r$ for large deviations for $f \in \fF^2_{r+1}$ in the range $(0,{\rm essup}(X))$.
\item[$(b)$] $S_N$ admits the strong asymptotic expansion of order $r$ for large deviations in the range $(0,{\rm essup}(X))$.
\end{enumerate}
 
\end{thm}

\begin{proof}
Take $\Ban=\reals$, $\ell=\text{Id}$ and $v=1$. Define $\cL_{\theta+is}$ acting on $\Ban$ by $\cL_{\theta+is} u=\EXP(e^{(\theta+is)X}) \cdot u$. Then, by the independence of $X_n$, $\cL^n_{\theta+is}1=\EXP(e^{(\theta+is)X})^n=\EXP(e^{(\theta+is)S_n})$. Since the moment generating function is finite on $J$, $(\theta+is) \mapsto \cL_{\theta+is}$ is analytic on the strip $\{ z \in \complex |\ \text{Re}(z) \in J \} $. So we have \eqref{MainAssum} and (B1). The validity of (B2) is immediate because $\Ban$ is one-dimensional, $\lambda(\theta)=\EXP(e^{\theta X})>0$ for $\theta \in J$, and $\lambda(0)=1$. 

Take $F$ to be the distribution function of $X$. For $\theta \in J$, we define $Y_{X, \theta}$ to be a random variable with distribution function $G^{\theta}$ given by 
\begin{equation}\label{ExpoG}
G^\theta(y)=\frac{e^{y\theta}F(y)}{\mu(\theta)},\ \text{where}\ \mu(\theta)=\int e^{y\theta }dF(y).
\end{equation}
Since $X$ is non-lattice, and distribution of $Y_{X,\theta}$ has a positive density with respect that of $X$, we have $Y_{X, \theta}$ is also non-lattice.
Therefore, for each $s \neq 0$, %there exists $\epsilon(s)\in (0,1)$ such that 
\begin{align}\label{B3IID}
\frac{|\EXP(e^{(\theta+is)X})|}{\EXP(e^{\theta X})}=|\EXP(e^{isY_{X,\theta}})| < 1.
\end{align}
This is equivalent to (B3). 

Since $Y_{X,\theta}$ has a positive density with respect that of $X$, $Y_{X, \theta}$ also satisfies the Cram\'er's condition (see \cite[Lemma 4]{BaR}). %Since, $Y_{X,\theta}$ satisfies the Cram\'er's condition, 
Therefore, \eqref{B3IID} holds uniformly in $|s| \geq 1$. That is, there exist $\epsilon \in (0,1)$ such that $|\EXP(e^{isY_{X,\theta}})|\leq\epsilon<1$ for $|s| \geq 1$. Therefore, $|\EXP(e^{(\theta+is)X})^n| \leq \EXP(e^{\theta X})^n \epsilon^n$,  for $|s| \geq 1$. This gives (B4) for arbitrary $r_1$. 

To see that $[C]$ holds, observe that $$(\log \lambda(\theta))^{\prime\prime} = \frac{\EXP(X^2 e^{\theta X})\EXP(e^{\theta X})-\EXP(Xe^{\theta X})^2}{\EXP(e^{\theta X})^2}.$$ 
From the H\"older's inequality, $\EXP(Xe^{\theta X})^2 \leq \EXP(X^2 e^{\theta X})\EXP(e^{\theta X})$, and the equality does not occur because %$\frac{Xe^{\theta X/2}}{e^{\theta X/2}}=X$ 
$X$ is not constant. Hence,  $(\log \lambda(\theta))^{\prime\prime}>0$.
\end{proof}

This provides an alternative proof for existence of strong asymptotic expansions for large deviations in \cite[Theorem 2 (Case 1)]{BaR} for iid sequences satisfying Cram\'er's condition. We also recover, \cite[Theorem 1 (Case 1, 3)]{BaR}, which gives us the first term of the expansions for non-lattice iid sequences.

\begin{thm}
Let $X$ be a non--lattice centred random variable whose logarithmic moment generating function is finite in a neighborhood of $0$. Let $X_n$ be a sequence of iid copies of $X$. Then, $S_N$ admits the order $0$ strong expansion for large deviations in the range $(0,\rm{essup}(X))$. 
\end{thm}
\begin{proof}
To see this we only have to observe that (B1), (B2), (B3) and $[C]$ hold as long as $X$ is non-lattice (we used Cram\'er's condition only when we established (B4) in the previous proof). So the result follows from \Cref{FirstTerm}. Also, we note that $\ell(\Pi_{\theta}v)=1$ for all $\theta$. Thus, we recover the results in \cite{BaR} mentioned above. 
\end{proof}
\subsection{Compactly supported $l-$Diophantine iid random variables}\label{l-DiophVar}
A random variable $X$ is called $l-$Diophantine if there exist positive constants $s_0$ and $C$ such that 
\begin{equation}\label{l-Dioph}
|\EXP(e^{isX})|<1-\frac{C}{|s|^l},\ |s|>s_0.
\end{equation}

Equivalently, a random variable $X$ with distribution function $F$ is $l-$Diophantine if and only if there exists $C_1,C_2 >0$ such that for all $|x|>C_1$, 
\begin{equation}\label{l-Dioph1}
\inf_{y \in \reals} \int_{\reals} \{ ax+y \}^2 dF(a) \geq \frac{C_2}{|x|^l},
\end{equation}
where $\{ x\} := \text{dist}(x, \integers)$ (see \cite{Br}). 

%In particular, $X$ is non-lattice. Let $X_1, X_2$ be two iid copies of $X$. The symmetrized $X$, $X^*$, is the random variable $X^*=X_1-X_2$. We note that $X$ is $l-$Diophantine if $X^*$ is $l-$Diophantine. This is because $\EXP(e^{isX^*})=|\EXP(e^{isX})|^2$. In addition, assume that $X$ has all exponential moments. Let $X_n$ be a sequence of iid copies of $X$. 

%We define $\cL_{\theta+is}$ as in the previous section. Then $(\theta+is) \mapsto \cL_{\theta+is}$ is entire and (B1)--(B3) follows as before. To see that (B4) holds we need the following lemma. 

%\begin{lem} Let $X$ and $Y$ be two random variables such that $X$ is $l-$Diophantine. Further assume that $X$ absolutely continuous with respect to $Y$ with the associated density function in $L^1 \cap L^2$. 
%\end{lem}
%\begin{proof}

%\end{proof}
Now, we describe two interesting classes of $l-$Diophantine random variables. In Case I, we discuss an iid sequence of compactly supported and $l-$Diophantine with ($l \neq 0$) random variables, while in Case II we assume, in addition, that those random variables take finitely many values.

\subsubsection{Case I}\label{non-discreteLDP}
Let $X$ be compactly supported and $l-$Diophantine with ($l \neq 0$). Then, assuming supp $X \subseteq [c,d]$,
\begin{align*}
\int_{\reals} \{ ax+y \}^2 dG^\theta(a) &= \frac{1}{\int_{c}^d e^{\theta a} dF(a)}\int_c^d \{ ax+y \}^2 e^{\theta a} dF(a)\\ &\geq \frac{e^{\theta c}}{\int_{\reals} e^{\theta a} dF(a)} \int_c^d \{ ax+y \}^2\, dF(a)
\end{align*}   
where $G^\theta$ is as in \eqref{ExpoG}. Thus, from \eqref{l-Dioph1}, for all $|x|>C_1$,
$$ \inf_{y \in \reals} \int_{\reals} \{ ax+y \}^2 dG^\theta(a) \geq \frac{e^{\theta c}}{\int_{c}^d e^{\theta a} dF(a)}\frac{C_2}{|x|^l}.$$
So the random variable $Y_{X,\theta}$ with distribution function $G^\theta$ is also $l-$Diophantine.

\begin{thm}
Let $X$ be compactly supported and $l-$Diophantine with $(l \neq 0)$. Then,  
\begin{enumerate}
\item[$(a)$] For all $r$, $S_N$ admits the weak asymptotic expansion of order $r$ for large deviations for all $a \in (0, \rm{essup}(X))$ for $f \in \fF^q_{r+1,\alpha}$, where $q > \lfloor (r+1)l/2 \rfloor $, for a suitable $\alpha$ depending on $a$.
\item[$(b)$] For all $r< \lceil 2l^{-1}\rceil$,  $S_N$ admits the strong asymptotic expansion for large deviations of order $r$ in the range $(0, \rm{essup}(X))$.
\end{enumerate}
\end{thm}
\begin{proof}
Taking $\cL_{\theta+is}$ as in \Cref{iidCramer}, we can establish the condition $[C]$, (B1), (B2) and (B3) as in the $0-$Diophantine case. (B4) follows from the $l-$Diophantineness of $Y_{X,\theta}$. In fact, 
\begin{equation*}
\frac{|\EXP(e^{(\theta+is)X})|}{\EXP(e^{\theta X})}=|\EXP(e^{isY_{X,\theta}})|<1-\frac{K_\theta}{|s|^{l}},\ |s|>1,
\end{equation*}
and hence, it follows that whenever $1<|s|<n^{\frac{1-\epsilon}{l}}$, $$|\EXP(e^{(\theta+is)S_n})|= \EXP(e^{(\theta+is)X})^n  \leq \EXP(e^{\theta X})^n e^{-C_\theta n^{\epsilon}/2}$$ where $\epsilon \in (0,1)$ can be made arbitrarily small. So $r_1=(1-\epsilon)l^{-1}<l^{-1}$. 
\end{proof}

\subsubsection{Case II}\label{DiscFinAtm}
Let $X$ be a centred random variable taking values $a_1,\dots,a_d$ $(d \geq 3)$ with probabilities $p_1,\dots,p_d$, respectively. Then the logarithmic moment generating function $h(\theta)=\log \EXP(e^{\theta X})$ of $X$ is finite for all $\theta \in \reals$. Take $X_n$ to be a sequence of iid copies of $X$. 

Take $\ba=(a_1,\dots,a_d)$, $b_j=a_j-a_1$, for $j=2\dots d$ and $d(s)=\max_{j\in \{2,\dots d\}} \dist(b_j s, 2\pi \integers)$. Then $\ba$ is called $\beta$-Diophantine if there is a constant $C>0$ such that for $|s|>1$, $$ d(s)\geq \frac{C}{|s|^\beta}. $$
In the rest of this section we assume that $\ba = (a_1,\dots,a_d)$ is $\beta-$Diophantine. In fact, almost all $\ba$ are $\beta-$Diophantine provided $\beta>(d-1)^{-1}$ (see \cite{Sp}). Since $\ba$ is $\beta-$Diophantine, the characteristic function of $X$ satisfies
$$|\EXP(e^{isX})|<1-\frac{c}{|s|^{2\beta}},\ |s|>1$$ for some $c$. This follows from the following Lemma whose proof can be found in \cite{DF}).

\begin{lem}\label{IneqChar}
Let $X$ be a discrete random variable taking values $a_1, \dots a_d$ with probabilities $p_1,\dots, p_d$, respectively, and $d(s)$ be as defined above. Then there exists a positive constant $c$ such that
\begin{equation*}
|\EXP(e^{isX})|\leq 1-c d(s)^2.
\end{equation*}
\end{lem}

Now we prove the existence of asymptotic expansions for large deviations in this setting.

\begin{thm}
Let $\ba = (a_1,\dots,a_d)$ be $\beta-$Diophantine. Take $X_n$ to be a sequence of iid copies of $X$. For all $r$, $S_N$ admits the weak expansion of order $r$ for $a \in (0, \max\{a_1,
\dots, a_d\})$ for $f \in \fF^q_{r+1,\alpha}$, where $q>\lfloor (r+1)\beta \rfloor$, for suitable $\alpha$ depending on $a$. 
\end{thm}

\begin{proof}
We define $\cL_{\theta+is}$ as in \Cref{iidCramer}. Then the conditions (B1), (B2), (B3) and $[C]$ are immediate from \Cref{iidCramer}. 
%Now, we establish condition (B4) using ideas in \cite{DF}. 

Due to \Cref{IneqChar}, as a random variable, $X$ is $2\beta-$Diophantine. Since $Y_{X,\theta}$ has a positive density with respect that of $X$, $Y_{X,\theta}$ is $2\beta-$Diophantine for all $\theta \in \reals$ as in \Cref{non-discreteLDP}. That is for all $\theta$, there exists $c_{\theta}$ such that
\begin{equation*}
\frac{|\EXP(e^{(\theta+is)X})|}{\EXP(e^{\theta X})}=|\EXP(e^{isY_{X,\theta}})|<1-\frac{c_\theta}{|s|^{2\beta}},\ |s|>1.
\end{equation*}
Therefore $|\EXP(e^{(\theta+is)S_n})|= \EXP(e^{(\theta+is)X})^n  \leq \EXP(e^{\theta X})^n e^{-C_\theta n^{\epsilon}/2}$ when $1<|s|<n^{\frac{1-\epsilon}{2\beta}}$, where $\epsilon \in (0,1)$ can be made arbitrarily small. So (B4) holds with $r_1<\frac{1-\epsilon}{2\beta}<\frac{1}{2\beta}$. 
\end{proof}

\iffalse \begin{proof}
Fix $s$. Let $j$ be an index such that $|e^{ia_ks}-e^{ia_1s}|=|e^{ib_ks}-1|$ is maximal. Then, 
\begin{align*}
|\EXP(e^{isX})| &\leq \sum_{k \neq 1,j}p_k +|p_je^{ia_js}+p_1e^{ia_1s}|\\ &= 1-(p_j+p_1)+|p_je^{ib_js}+p_1|
\end{align*}
Next, notice that 
\begin{align*}
\bigg|\frac{p_j}{p_j+p_1}e^{ib_js}+\frac{p_1}{p_j+p_1}\bigg|^2 = 1 - \frac{p_jp_1}{(p_j+p_1)^2}|1-e^{ib_js}|^2.
\end{align*}
Since $|1-e^{ib_js}|^2 \leq 4$, 
\begin{align*}
\bigg|\frac{p_j}{p_j+p_1}e^{ib_js}+\frac{p_1}{p_j+p_1}\bigg| \leq  1 - \frac{p_jp_1}{4(p_j+p_1)^2}|1-e^{ib_js}|^2.
\end{align*}
Therefore, 
\begin{align*}
|\EXP(e^{isX})| &\leq 1 - \frac{p_jp_1}{4(p_j+p_1)}|1-e^{ib_js}|^2.
\end{align*}
$d(s)=\max_{k\in \{2,\dots d+1\}} \dist(b_k s, 2\pi \integers)$ implies that
\begin{align*}
|\EXP(e^{isX})| &\leq 1 - \frac{p_jp_1}{4(p_j+p_1)}|1-e^{id(s)}|^2.
\end{align*}
Finally, since $0 \leq d(s)<2\pi$, $|1-e^{id(s)}| \geq \tilde{c}d(s)$ for some $\tilde{c}$, and hence, taking $c=\tilde{c}^2\min_{k \neq 1}\frac{p_kp_1}{4(p_k+p_1)}$ we have
$|\EXP(e^{isX})|\leq 1 - cd(s)^2.$
\end{proof}
\fi
%Now we can apply our results: 

However, one can show that strong expansions of order $2d-3$ or higher do not exist. To see this, let $\wt{S}_n$ be sum of $n$ iid copies of $Y_{X,\theta}$ (defined in \Cref{iidCramer}). Note that $\wt{S}_n$ takes $\cO(n^{d-1})$ different values. Therefore, $\Prob(\wt{S}_n > an)$ has jumps of order $\cO(n^{-(d-1)})$. As a result, as $\epsilon \to 0$, $\Prob(\wt{S}_n > (a+\epsilon)n)$ and $\Prob(\wt{S}_n > (a-\epsilon)n)$ may differ only by at most $\cO(n^{-(d-1)})$. This forces the order of the strong asymptotic expansion to satisfy $\frac{r+1}{2}<d-1 $, which gives us $r < 2d-3$, as required. Thus, this is an example where weak expansions exist even when strong expansions fail to exist. 

\begin{comment}
Choose $\beta < \frac{1}{d-3}$. If $X$ is $2\beta-$Diophantine then $Y_{X,\gamma}$ admits strong Edgeworth expansion of order $d-2$ and lower. Therefore, $Y_{X,\gamma}$ admits weak local Edgeworth expansions of order $d-3$ and lower on $F^1_0$. This means that $X_n$ admits strong asymptotic expansions for large deviations of order $d-3$ and lower in the range $(0,\sup(\supp\ X))$.
 
{\color{blue} Note that one cannot expect the strong expansions of order $d-2$ or higher to exist because $Y_{X,\gamma}$ has $d$ atoms.}

However, we can show that weak expansions of higher orders exist. Because $Y_{X,\gamma}$ is $2\beta-$Diophantine, given $a \in (0,\sup(\supp\ X))$, choosing $q>\beta(r+2)+1$ and $\alpha > \theta_a$, for every $f \in F^{q}_{r+1,\alpha}$ the order $r$ weak expansion for large deviation exists.
\end{comment}

\subsection{Time homogeneous Markov chains with smooth density}\label{MarkovDensity}

Take $x_n$ to be a time homogeneous Markov process on a compact connected manifold $M$ with $C^1$ transition density $p(x,y)$, which is bounded away from $0$ (non-degenerate). Let $X_n=h(x_{n-1}, x_n)$ for a $C^1$ function $h:M\times M\to\reals$.  We assume that $h(x,y)$ can not be written in the form
\begin{equation}\label{CoB}
h(x,y)= H(y)-H(x) + c(x,y),
\end{equation}
where $H\in L^\infty(M)$ and $c(x,y)$ is lattice valued.
The following lemma characterizes such $h$ (see \cite{FL}).
\begin{lem}
\label{LmMc-CoB}
\eqref{CoB} holds iff there exists $o\in M$ such that the function $x \mapsto h(o,x)+h(x,y)$ is  lattice valued.
% ({\color{red} Piece--wise constant functions means finitely many pieces?})
\end{lem}
Note that the CLT holds for $X_n$ and the limiting normal distribution is degenerate if and only if \eqref{CoB} holds with constant $c(x,y)$ (see \cite{GuH}). Therefore, in our setting, the CLT is non-degenerate.

We need the following lemma to obtain the condition $[B]$. 
\begin{lem}\label{PosCptOp}
Let $K(x,y)$ be a positive $C^k$ function on $M \times M$. Let $P$ be an operator on $L^\infty(M)$ given by $$Pu(x)=\int_{\cM} K(x,y) u(y) \, dy.$$ Then $P$ has a simple leading eigenvalue $\lambda >0$, and the corresponding eigenfunction $g$ is positive and $C^k$. 
\end{lem}
\begin{proof}
From the Weierstrass theorem, $K(x,y)$ is a uniform limit of functions formed by finite sums of functions of the form $J(x)L(y)$. Therefore, $P$ can be approximated by finite rank operators. So $P$ is compact on $L^\infty(M)$. Since $P$ is an operator that leaves the cone of positive functions invariant, by a direct application of Birkhoff Theory (see \cite{Bi}), $P$ has a leading eigenvalue $\lambda$ that is positive and simple along with a unique positive eigenfunction $g$ with $\|g\|_{\infty} = 1$.
 
Since, $\lambda g(x) = \int_M K(x,y)g(y) dy$ and $K(x,y)$ is $k$ times continuously differentiable in $x$ and $M \times M$ is compact, we can differentiate under the integral sign $k$ times. This means $g$ is $C^k$. 

\end{proof}

The next theorem establishes the existence of strong and weak expansions for large deviations in this setting. 
\begin{thm}\label{MarkovPrcss}
Take $x_n$ to be a time homogeneous Markov chain on a compact connected manifold $M$ with $C^1$ non-degenerate transition density $p(x,y)$. Let $X_n=h(x_{n-1}, x_n)$ for a $C^1$ function $h:M\times M\to\reals$ that does not satisfy \eqref{CoB}. Take $B=\lim_{n \to \infty}\frac{B_n}{n}\ \text{with}\ B_n= \sup\{\sum_{j=1}^n h(x_{j-1},x_j)| x_i \in M, 1\leq i \leq n \}$. Then, for all $r$, 
\begin{enumerate}
\item[$(a)$] $S_N$ admits the weak asymptotic expansion of order $r$ for large deviations in the range $(0,B)$, for $f \in \fF^q_{r+1,\alpha}$ with $q \geq 1$ and suitable $\alpha$ depending on $a$.
\item[$(b)$] $S_N$ admits the strong asymptotic expansion of order $r$ for large deviations in the range $(0,B)$.
\end{enumerate}
 
\end{thm}
\begin{proof}
Take $\Ban=L^\infty(M)$ and consider the family of integral operators,
$$(\cL_z u)(x)=\int_{\cM} p(x,y) e^{z h(x,y)} u(y)\, dy,\ z \in \complex.$$
Let $\mu$ be the initial distribution of the Markov chain. Then, using the Markov property, we have $\EXP_{\mu}[e^{zS_n}]=\mu(\cL^n_z 1).$ Now we check the condition $[B]$.

It is straightforward that $z \mapsto \cL_z$ is entire and therefore (B1) holds. Note that, for all $\theta$, $\cL_\theta$ is of the form $P$ in \Cref{PosCptOp}. Therefore, (B2) holds for all $\theta$. Take $\lambda(\theta)$ to be the top eigenvalue and $g_\theta$ to be the corresponding eigenfunction. Then $g_\theta$ is $C^1$. 

To show (B3) and (B4), we define a new operator $Q_\theta$ as follows. 
$$(Q_{\theta} u )(x) = \frac{1}{\lambda(\theta)}\int_{\cM} e^{\theta h(x,y)} p(x,y) u(y) \frac{g_\theta(y)}{g_\theta(x)}\, d(y).$$
It is easy see to that $$ p_{\theta} (x,y)=\frac{e^{\theta h(x,y)} p(x,y)}{g_{\theta} (x) \lambda(\theta)}\ \text{and}\ dm_{\theta}(y) =g_{\theta} (y)\, dy $$ define a new Markov chain $x^\theta_n$ with the associated Markov operator $ Q_\theta $. Observe that $Q_\theta$ is a positive operator and $Q_\theta 1 = \frac{1}{\lambda(\theta)}\int_{\cM} e^{\theta h(x,y)} p(x,y) \frac{g_\theta(y)}{g_\theta(x)}\, dy = 1 $ (since $g_\theta$ is the eigenfunction corresponding to eigenvalue $\lambda(\theta)$ of $\cL_\theta$). 

Now we can repeat the arguments in \cite{FL} to establish the properties of the perturbed operator given by $$(Q_{\theta+is})u (x) = \int_{\cM} e^{ish(x,y)} p_\theta (x,y) \, dm_{\theta}(y) $$
Since \eqref{CoB} does not hold, %there is no $H$ such that $h(x,y)+H(y)-H(x)$ is piece-wise constant, 
we conclude that sp$(Q_{\theta+is}) \subset \{|z|<1\}$ (see \cite[Section 3.6.3]{FL}). 

Take $G_\theta$ to be the operator on $L^\infty(M)$ that corresponds to multiplication by $g_\theta$. Then $\cL_{\theta+is}=\lambda(\theta) G_\theta \circ Q_{\theta+is} \circ G^{-1}_\theta$. Therefore, sp$(\cL_{\theta+is})$ is the sp$(Q_{\theta+is})$ scaled by $\lambda(\theta)$. This implies sp$(\cL_{\theta+is}) \subset \{|z|< \lambda(\theta) \}$ as required. 

Also, since $g_\theta$ is $C^1$, we can integrate by parts, as in \cite[Section 3.6.3]{FL}, to conclude that there exist $\epsilon_\theta \in (0,1)$ and $r_\theta > 0$ such that $\|Q^2_{\theta+is}\| \leq (1-\epsilon_\theta)$ for all $|t| \geq r_\theta$. Therefore, $$\|\cL^n_{\theta+is}\|=\lambda(\theta)^n\|G_\theta Q^n_{\theta+is}G^{-1}_\theta \| \leq \lambda(\theta)^n\|G_\theta\|\|Q^n_{\theta+is}\|\|G^{-1}_\theta \| \leq  C\lambda(\theta)^n(1-\epsilon_\theta)^{\lfloor n/2\rfloor}.$$ 

Now we establish $[C]$. Since \eqref{CoB} does not hold, the asymptotic variance $\sigma^2_\theta$ of $X^\theta_n=h(x^\theta_{n-1},x^\theta_n)$ is positive. Taking $\gamma(\theta+is)$ to be the top eignevalue of $Q_{\theta+is}$, $\lambda(\theta+is)=\lambda(\theta)\gamma(\theta+is)$. Thus, 
\begin{align*}
(\log \lambda(\theta))^{\prime \prime} =-\frac{d^2}{ds^2} \log \lambda(\theta+is)\Big|_{s=0} &= -\frac{d^2}{ds^2}  \log \gamma(\theta+is) \Big|_{s=0} \\ &=-\frac{\gamma^{\prime \prime}(\theta)}{\gamma(\theta)}+\Big(\frac{\gamma^{\prime}(\theta)}{\gamma(\theta)}\Big)^2=-\gamma^{\prime \prime}(\theta)+\gamma^{\prime}(\theta)^2.
\end{align*}
Put $S^\theta_N=X^\theta_1+\dots+X^\theta_N$. Since $\EXP(e^{isS^\theta_N})=\int Q^N_{\theta+is} 1 \, d\mu$, from \eqref{derivatives} we have that $\gamma^{\prime}(\theta)^2-\gamma^{\prime \prime}(\theta)=\sigma^2_\theta$. Thus, $(\log \lambda (\theta))^{\prime \prime}=\sigma^2_\theta >0$. %Therefore, $\log \lambda (\theta)$ is a strictly convex function. 

Note that $\cL_\theta = \lambda(\theta) \Pi_\theta + \Lambda_\theta$, where $\Pi_\theta$ is the projection onto the top eigenspace. From \cite[Chapter III]{HH}, $\Pi_\theta = g_\theta \otimes \varphi_\theta$, where $\varphi_\theta$ is the top eigenfunction of $Q^*_\theta$, the adjoint of $Q_\theta$. Since $Q^*_\theta$ itself is a positive compact operator acting on $(L^\infty)^*$ (the space of finitely additive finite signed measures), $\varphi_\theta$ is a finite positive measure. Hence, $\mu(\Pi_\theta 1) = \varphi_\theta(1) \mu(g_\theta) >0$ for all $\theta$.

The rate function $I(a)$ is finite for $a\in (0,B)$, where $B = \lim_{ \theta \to \infty} \frac{\log \lambda(\theta)}{\theta}$. We observe that $B<\infty$ because $h$ is bounded, i.e.,\hspace{3pt}$\frac{S_n}{n} \leq \|h\|_\infty$. In fact, $$B=\lim_{N \to \infty}\frac{B_n}{n}\ \text{with}\ B_n= B=\lim_{n \to \infty}\frac{B_n}{n}\ \text{with}\ B_n= \sup\{\sum_{j=1}^n h(x_{j-1},x_j)| x_i \in M, 1\leq i \leq n \}.$$  

To see this, note that $B_n$ is subadditive. So $\lim_{n \to \infty}\frac{B_n}{n}$ exists and is equal to $\inf_{n} \frac{B_n}{n}$. Given $a>B$, there exists $N_0$ such that $\frac{S_n}{n}\leq \frac{B_n}{n}<a$  for all $n>N_0$. Thus $\Prob (S_n \geq an)=0$ for all $n>N_0$, and hence $I(a)=\infty$. Next, given $a<B$, for all $n$, $B_n>an$. Fix $n$. Then there exists a realization $x_1,\dots,x_n$ such that $an < \sum h(x_{j-1},x_j) \leq B$. Since $h$ is uniformly continuous on $M \times M$, there exists $\delta>0$ such that by choosing $y_j$ from a ball of radius $\delta$ centred at $x_j$, we have $an < \sum h(y_{j-1},y_j) \leq B$. We estimate the probability of choosing such a realization $y_1,\dots,y_n$ and obtain a lower bound for $\Prob(S_n \geq an)$:
\begin{align*}
\Prob(S_n \geq an) &\geq \int_{\mathbb{B}(x_n,\delta)}\dots\int_{\mathbb{B}(x_1,\delta)} \int_{\mathbb{B}(x_0,\delta)} p(y_{n-1},y_n) \dots p(y_0,y_1)\, d\mu(y_0)\, dy_1 \dots \, dy_N \\ &\geq { \mu(\mathbb{B}(x_0,\delta))} \Big(\min_{x,y \in \cM} p(x,y)\Big)^n \text{vol}(\mathbb{B}_\delta)^n.
\end{align*}
Therefore, $I(a)< \infty$, as required.
\end{proof}

\subsection{Finite State Markov chains}\label{FiniteMarkov}

Consider a time homogeneous Markov chain $x_n$ with state space $S=\{1,\dots,d\}$ whose transition probability matrix $P=(p_{jk})_{d\times d}$ is positive. Suppose that $\bh=(h_{jk})_{d\times d} \in $ M$(d,\reals)$ is such that there are no constants $c, r$ and a $d-$vector $H$ such that 
\begin{equation}\label{NoCo-B}
r h_{jk}=c+H(k)-H(j) \mod\ 2\pi
\end{equation}
for all $j,k$. Define $X_n=h_{x_nx_{n+1}}$.

Next, define $b_{l,j,k}=h_{lj}+h_{jk}$, $l,j,k \in \{1,\dots,d\}$ and $$d(s)=\max_{l,j,k}\ \text{dist}((b_{l,j,k}-b_{l,1,k})s, 2\pi \integers).$$ 
We further assume that $\bh$ is $\beta-$Diophantine, that is, there exists $K \in \reals$ such that for all $|s|>1$, 
\begin{equation}\label{DiophCon}
d(s)\geq \frac{K}{|s|^\beta}.
\end{equation}
If $\beta>\frac{1}{d^3-d^2-1}$, then almost all $\bh$ are $\beta-$Diophantine (see \cite{Sp}). These assumptions yield the following result.

\begin{thm}
Take $x_n$ to be a time homogeneous Markov chain on \{1,\dots,d\} with a positive transition probability matrix $P=(p_{jk})_{d\times d}$. Let $X_n=h_{x_{n}x_{n+1}}$ for $\bh$ that does not satisfy \eqref{NoCo-B} and is  $\beta-$Diophantine. Take $B=\lim_{n \to \infty}\frac{B_n}{n}$, $B_n= \sup\{\sum_{j=1}^n \bh_{x_{j-1}x_j}| x_0,\dots,x_n \in S\}$. Then, for all $r$, $S_N$ admits the weak expansion of order $r$ in the range $(0, B)$, for $f \in \fF^q_{r+1,\alpha}$ where $q>\lfloor (r+1)\beta \rfloor$, and for suitable $\alpha$ depending on $a$.
\end{thm}
 
\begin{proof}
We use ideas from the previous section and \cite{FL} to establish conditions $[B]$ and $[C]$. 

Consider the family of operators $\cL_{\theta+is} : \complex^d \to \complex^d$,
\begin{equation}\label{TransProb}
(\cL_{\theta+is}f)_j = \sum_{k=1}^d e^{(\theta+is)h_{jk}}p_{jk}f_k,\ j=1,\dots,d.
\end{equation}
Take $v=\bf{1}$ and $\ell=\bmu_0$, the initial distribution. Then, from the Markov property, we obtain $\EXP(e^{isS_N})=\sum_{j=1}^d(\bmu_0)_j(\cL^n_{\theta+is} {\bf 1})_j=\bmu_0(\cL^n_{\theta+is} {\bf 1})$. Obviously, $z \mapsto \cL_z$ is entire. So \eqref{MainAssum} and (B1) hold. 

The matrix $P^\theta=(e^{\theta h_{jk}}p_{jk})$ is positive for each $\theta$, and hence, by the Perron-Frobenius theorem, $P^\theta$ has a positive leading eigenvalue $\lambda(\theta)$ that is simple, and the corresponding eigenvector $\bg^\theta=(g^\theta_j)$ is positive. In addition, $P$ (corresponding to $\theta=0$) is stochastic. So its top eigenvalue satisfies $\lambda(0)=1$. Since we deal with finite-dimensional spaces, the remaining part of (B2) follows immediately. 

Next, define a new Markov chain $x^\theta_n$ corresponding to the stochastic matrix 
\begin{equation}\label{TiltedMatrix}
\overline{P}^\theta=\Big(\frac{e^{\theta h_{jk}}p_{jk}g^\theta_k}{\lambda(\theta)g^\theta_j}\Big).
\end{equation}
Then the corresponding operator is
\begin{equation*}
(Q_{it}f)_j = \sum_{k=1}^d e^{ith_{jk}}\frac{e^{\theta h_{jk}}p_{jk}g^\theta_k}{\lambda(\theta)g^\theta_j}f_k,\ j=1,\dots,d.
\end{equation*}

Also, (B3) follows because \eqref{NoCo-B} does not hold. For a proof of this fact refer to \cite[Section 3.6.2]{FL}, where (B3) is proven for $\theta=0$. From the Diophantine condition \eqref{DiophCon}, there exists $c >0$ such that
$\|\cL^2_{is}\| \leq 1- cd(s)^2$. For a proof of this, refer to \cite[Section 3.6.2]{FL}. So $$\|\cL^n_{is}\|  \leq \big(1- c d(s)^2  \big)^{\lceil n/2 \rceil} \leq e^{-C  s^{-2\beta} n/2}\ \text{for}\ |s|>1.$$ 
Thus, $\|\cL^n_{is}\|  \leq e^{-C n^{\epsilon}/2}$ when $1<|s|<n^{\frac{1-\epsilon}{2\beta}}$. 
Note that the Diophantine nature of the matrix $\bh$ is independent of the change of measure done in  \eqref{TiltedMatrix} and hence, the underlying Markov process. Therefore, the same proof applies to $\theta \neq 0$.  %because it relies only on the properties of $\bh$, not on the underlying Markov process. 
%{\color{red}Similarly, it can be easily shown that the argument for $\theta =0$ (original chain) works for all $\theta$ and therefore $\|Q^n_{is}\|\leq e^{-C n^{\epsilon}/2}$ for $1<|s|<n^{\frac{1-\epsilon}{2\beta}}$. }%This is because the decay depends on the Diophantine nature of the matrix $\bh$, not on the underlying Markov process. 

Note that $\cL_{\theta+is} = \lambda(\theta) G_\theta \circ Q_{is} \circ G^{-1}_\theta$, where $G_\theta$ corresponds to multiplication by $G_\theta=(g^\theta_j\delta_{jk})$. Therefore, $\|\cL^n_{\theta+is}\|\leq C\lambda(\theta)^se^{-C n^{\epsilon}/2}$ for $1<|s|<n^{\frac{1-\epsilon}{2\beta}}$, which gives us (B4) with $r_1=\frac{1-\epsilon}{2\beta}$, where $\epsilon>0$ can be made arbitrarily small. 

The same argument as in in \Cref{MarkovPrcss} adapted to finite state chains gives us condition $[C]$ and the fact that the range of large deviations is $(0,B)$, where 
\begin{equation}\label{LDrange}
B= \lim_{n \to \infty} \frac{B_n}{n}\ \text{with}\ B_n= \sup\Big\{\sum_{j=1}^n \bh_{x_{j-1}x_j}| x_0, \dots x_n \in S  \Big \}.
\end{equation}
\end{proof}

%As a result, this case is similar to the discrete iid case. Define, 
However, as in the case of discrete iid random variables, strong expansions of order $2d^2-3$ or higher do not exist because the number of distinct values $X_n$ takes is at most $d^2$.

Note that in the proof of the previous theorem, the Diophantine nature of $\bh$ was not used in proving (B1), (B2), (B3) and [C]. Therefore, we also have the following first order asymptotics for large deviations for a general finite state Markov chain. 

\begin{thm}
Take $x_n$ to be a time homogeneous Markov chain on \{1,\dots,d\} with a positive transition probability matrix $P=(p_{jk})_{d\times d}$. Let $X_n=h_{x_{n}x_{n+1}}$ for $\bh$ that does not satisfy \eqref{NoCo-B}. Then $S_N$ admits the order $0$ strong expansion for large deviations in the range $(0,B)$ where $B$ is as in \eqref{LDrange}.
\end{thm}
%Because for small $\epsilon$, $\lceil \frac{r+1}{2(1-\epsilon)} \rceil= \lceil \frac{r+1}{2} \rceil$ take $q>\frac{r+1}{2}\beta$. Then, given $a\in (0,\sup h_{jk})$ and $\alpha>\theta_a$, $X^{\bh}_n$ admits weak expansions for large deviations of order $r$ for $f \in F^{q+2}_{r+1,\alpha}$. 
\subsection{Sub-shifts of finite type}\label{SFTs}
In this section, we prove an exact Large Deviation Principle for subshifts of finite type (SFT's). Many concrete dynamical systems like Axiom A diffeomorphisms and Markov maps of the interval can be studied by converting them to SFT's via a symbolic coding.  See, for instance, \cite{PP} for a multitude of such examples. Hence, the exact Large Deviation Principle we establish here, applies beyond the setting in which it is introduced. 

We recall some facts about SFT's without proof. \cite[Chapters 1--4]{PP} contain a detailed account of the theory as well as proofs of the following.

Let $A$ be a $k \times k$ matrix with only $0$ and $1$ as entries. Define  
$$\Sigma^+_A=\Big\{ \vec{x}=(x_j) \in \Sigma^+=\prod_{j=0}^{\infty} \{1,2,\dots,k\}\, \Big|\, A(x_j,x_{j+1})=1, \ \forall j \in \naturals_0  \Big\}.$$
Let $\sigma : \Sigma^+_A \to \Sigma^+_A$ act on a sequence by truncating the first symbol and  moving remaining elements to the left by one position, i.e., $\sigma\big((x_n)_{0}^\infty\big)=(x_{n+1})_{0}^\infty$. Then, $(\Sigma^+_A,\sigma)$ is called a subshift of finite type (also known as a topological Markov chain).
%In fact, every subshift of finite type can be obtained in this way. 
Define the period $d$ of $A$ by $d = \gcd \{n\ |\ \exists j, A^n_{jj}>0 \}$ and if $d=1$, $A$ is called aperiodic. Also, $A$ is called irreducible if for all $i,j$ there exists $N$ such that $A^N_{ij}>0$.

The Tychonoff product topology on $\Sigma^+$ is metrizable. Let $\epsilon \in (0,1)$. Then, a metric on $\Sigma$ can be defined by $\dd(\vec{x},\vec{y})=\epsilon^N$ where $N \in \naturals_0$ is the maximal $N$ such that $x_j=y_j$ for all $|j|<N$. This induces the product topology on $\Sigma^+$ and we consider its restriction to $\Sigma^+_A$. 

Let $f:\Sigma^+_A \to \complex$ be continuous and $\text{var}_n(f)=\sup\big\{|f(\vec{x})-f(\vec{y})|\, \big|\, x_j=y_j,\, \forall  j \leq n\big\}$.
%allows us to characterize H\"older continuous functions on $\Sigma$. 
Take $H^{C,\alpha}$ to be the $\alpha-$H\"older functions with H\"older constant $C$.  Then,
$f \in H^{C,\alpha}$ if and only if $\text{var}_n(f) \leq C \epsilon^{n\alpha}$ for all $n\in \naturals_0$.
In particular, this characterizes the space of Lipschitz functions (corresponds to $\alpha=1$) on $\Sigma^+_A$ which is denoted by $F^+_\epsilon$. Define,
$
|f|_\infty = \sup \big\{|f(\vec{x})|\ \big|\ \x \in\Sigma \big\}$, 
$|f|_\epsilon= \sup \big\{\epsilon^{-n}\text{var}_n f \ \big|\ n \in \naturals_0 \big\}$ and $\|f\|_\epsilon = |f|_\infty+ |f|_\epsilon$. 
Then, $(F^+_\epsilon,\|\cdot\|_\epsilon)$ is a Banach space such that $\|\cdot\|_\epsilon-$bounded sets are $|\cdot|_\infty-$compact. 

From now on we focus only on $\reals-$valued functions in $F^{+}_\epsilon$. 
A function $f\in F^+_\epsilon$ is called a coboundary if there exist $g \in F^+_\epsilon$ such that $f = g \circ \sigma - g$, and it is said to be generic if the only solution to $F(\sigma(\x)) = e^{itf(\x)}F(\x)$ in $F^+_\epsilon$ is a constant $F$ and $t=0$. Note that if $f$ is a coboundary then it is not generic. Given $f$ and $g$, we say $f$ and $g$ are cohomologous if $f-g$ is a coboundary.  

Define the pressure of $f$ by
$$P(f)=\sup_{\mu \in \cM^1_\sigma}\Big \{h_\mu(\sigma)+\int f\, d\mu\Big \}$$
where $h_\mu(\sigma)$ is the entropy of $\sigma$ with respect to $\mu$ and $\cM^1_\sigma$ is the space of $\sigma-$invariant probability measures. Then, there is a unique $\sigma-$invariant probability measure $m$ such that $P(f)=h_m(\sigma)+\int f\, dm$, and this $m$ is called the stationary equilibrium state of $f$, and $f$, a potential of $m$. It follows that $P(f+c)=P(f)+c$, and if $f$ and $g$ are cohomologous then $P(f)=P(g)$. Given a stationary equilibrium state $m$ and any two potentials $f,g$ of $m$, there is a constant $c$ such that $f-g$ is cohomologous to $c$. We call $f$ a normalized potential of $m$ if $f$ is a potential of $m$ and $P(f)=0$. In fact, this potential $f$ is unique upto a coboundary.

Now, we state and prove a strong large deviation result for irreducible, aperiodic SFT's.

\begin{thm}
Suppose $(\Sigma^+_A, \sigma)$ is a subshift of finite type with an irreducible, aperiodic $A$. Let $g \in F^+_\epsilon$ be $\reals-$valued. Suppose $m$ is the stationary equilibrium state of $g$ and it is normalized.  Let $f \in F^+_\epsilon$ be $\reals-$valued. Suppose $f \in F^+_\epsilon$ is generic, not cohomologous to a constant and $\int f \, dm = 0$. Define $X_n=f \circ \sigma^{n-1}$, $n \geq 1$ with initial distribution $m$, $$B=\lim_{\theta \to \infty} \frac{P(g+\theta f)}{\theta}=\sup_{\mu \in \cM^1_\sigma} \int f \, d\mu,$$
and 
$$I(a)=\sup_{\theta \geq 0} \Big \{ a \theta - P(g+\theta f) \Big \}$$
Then, for all $a \in (0,B)$, there exists a constant $K=K(a)$ such that 
$$\Prob(S_N \geq aN) e^{I(a)N} =  \frac{K}{\sqrt{2\pi N}}\Big(1 + o(1)\Big)\,\,\,\,{\text as}\,\,N \to \infty.$$
\end{thm}

\begin{proof} 
We introduce the family of Ruelle operators $L_{g+zf}:F^+_\epsilon \to F^+_\epsilon $, $z \in \complex$, 
$$\cL_{g+zf}(w)(\x)=\sum_{\sigma(\y)=\x} e^{g(\y)+zf(\y)}w(\y).$$
We establish the conditions (B1), (B2), (B3) and $[C]$ for this family of operators. Then, the result follows from \Cref{FirstTerm}.

It is straightforward that these are bounded linear operators, and  that $z \mapsto \cL_{g+zf}$ is analytic. Also,
$$\EXP_m\big(e^{zS_n}\big)=\int e^{zS_n}\, dm= \int \cL^n_{g+zf}{\bf 1}\, dm = m \big(\cL^n_{g+zf}{\bf 1}\big),\ z\in \complex .$$
From Ruelle-Perron-Forbenius Theorem (\cite[Theorem 2.2]{PP}), for all $\theta \in \reals$, $\cL_{g+\theta f}$ has a simple maximal positive eigenvalue $\lambda(\theta)$ given by $\lambda(\theta)=e^{P(g+\theta f)}$ with a positive eigenfunction $h_\theta$ and the rest of its spectrum is contained strictly inside $\{|z|<\lambda(\theta)\}$. Also, $\lambda(0)=e^{P(g)}=1$ by the choice of normalized potential $g$. This is (B2).

Since $f$ is generic, for all $t \neq 0$, $\cL_{g+(\theta+it)f}$ does not have eigenvalues on $\{z \in \complex\, |\, |z|=e^{P(g+\theta f)}\} $. This follows from the remarks appearing before the Theorem 4.13 in \cite{PP}. From \cite[Theorem 4.5]{PP} if $\cL_{g+(\theta+it)f}$ does not have an eigenvalue of modulus $e^{P(f+\theta g)}$ then its spectral radius is strictly smaller than $e^{P(g+\theta f)}$. This establishes (B3).

Again, from the Ruelle--Perron--Forbenius Theorem, the projection of to the top eigenspace $\Pi_\theta$ takes the form $h_\theta \otimes \mu_\theta$ where $h_\theta\, d\mu_\theta$ is the equilibrium state of $g+\theta f$. %, and $\int h_\theta \, d\mu_\theta = 1$. 
Hence, $m(\Pi_\theta {\bf 1})=\int h_\theta \, dm >0$. Also, from \cite[Proposition 4.12]{PP}, $P''(\theta)>0$ if and only if $f$ is not cohomologous to a constant. Therefore, we have $[C]$. 
\end{proof}

\subsection{Smooth Expanding Maps}\label{ExpandingMaps}
Uniformly expanding maps are the most basic type of uniformly hyperbolic systems, and as a result they have been studied extensively. Most of their statistical properties are well-known. See, for example, \cite{G} and references therein. Here we establish an exact Large Deviation Principle for $C^1-$observables in the setting of $C^2-$expanding maps of the torus. 

Suppose $f$ is smooth and uniformly expanding on $\mathbb{T}$, i.e., $f \in C^{r}(\mathbb{T},\mathbb{T})$, $r\geq 2$ and there is $\lambda_*$ such that $\inf_{x \in \mathbb{T}}|f'(x)|\geq \lambda_*>1$.  Let $g\in C^{r-1}(\mathbb{T},\reals)$ be such that there is constant $c$ and $\phi \in C^0(\mathbb{T}, \reals)$ such that
\begin{equation}\label{CtsCoB}
g=c+\phi - \phi \circ f. 
\end{equation}
That is, $g$ is not a continuous coboundary. Take $X_n= g\circ f^{n-1}$, $n \geq 1$. If we choose an initial point $x$ according to a probability density $\rho(x)$ then $\{X_n\}$ becomes a sequence of random variables. WLOG assume $\int_{\Tor} g(x)\rho(x)\, dx = 0$. 

Then, the following theorem establishes a strong large deviation result for $X_n$. 

\begin{thm}
 Suppose $f \in C^{r}(\mathbb{T},\mathbb{T})$, $r \geq 2$ and uniformly expanding on $\mathbb{T}$. Let $g \in C^{r-1}(\mathbb{T},\reals)$ be such that \eqref{CtsCoB} does not hold. Take $X_n=g \circ f^n$ with initial distribution $\mu$. Define $\cM^1_f(\mathbb{T})=\{\nu \in \cM^1(\mathbb{T})| f_*\nu=\nu\}$, $B=\sup_{\nu \in \cM^1_f(\mathbb{T})} \int g\, d\nu,$
and
\begin{equation}\label{ExpandRate}
I(a)=-\sup_{\nu \in \cM(a)} \bigg[{h}_{KS}(\nu)-\int (\log f')\, d\nu \bigg]
\end{equation}
where $\cM(a)=\{\nu \in \cM^1(\mathbb{T})\ |\ f_*\nu=\nu, \int g\ d\nu = a \}$ and $h_{KS}$ is the Kolmogorov-Sinai entropy. Then, for all $a \in (0,B)$ there exists a constant $K=K(a)$ such that
\begin{equation}\label{Expanding}
\Prob(S_N \geq aN) e^{I(a)N} =  \frac{K}{\sqrt{2\pi N}}\Big(1 + o(1)\Big)\,\,\,\,{\text as}\,\,N \to \infty.
\end{equation}

%is the collection of all $f-$invariant probability measures $\nu$ such that $\int g\, d\nu = a$.   
\end{thm}

\begin{proof}
Take $\cL$ to be the transfer operator associated with $f$, $$\cL(h)(x)=\sum_{f(y)=x} \frac{h(y)}{f'(y)}.$$ For $z \in \complex$, define $\cL_{z}: C^1 \to C^1$ by $\cL_z(\cdot)=\cL(e^{zg}\ \cdot\ )$. That is,
$$\cL_z(h)(x)=\sum_{f(y)=x} e^{zg(y)}\frac{h(y)}{f'(y)}.$$
Then, it follows from properties of the transfer operator that
$$\EXP(e^{zS_n})=\int (\cL^n_{z}\rho)(x)\, dx.$$
Also, $z \mapsto \cL_z$ is analytic due to the power series expansion, $\cL_z(\cdot)=\sum_{k=0}^\infty \frac{z^k}{k!} \cL(g^k \ \cdot\ )$. Note here that, $\cL(g^k\ \cdot\ ): C^1 \to C^1$ because $\|g\|_\infty<\infty$ and $\|g'\|_\infty < \infty$. 

From \cite[Lemma A.1]{DL}, we have that for $\theta \in \reals$, $\cL_{\theta}$ is of Perron-Forbenius type for all $\theta$ and the projection operator to the top-eigenspace $\Pi_{\theta}$ takes the form $h_\theta \otimes m_{\theta}$ where $h_\theta \in C^1$ is positive and $m_\theta$ is a positive measure. That is for all $\theta$, $\cL_\theta=\lambda(\theta)h_\theta \otimes m_{\theta} + \Lambda_\theta$ with $\|\Lambda_\theta\| < Cr^n_\theta$ where $0<r_\theta<\lambda(\theta)$. %It is also, well-known that sp$(\cL)=\{|z|\leq \epsilon <1\} \cup \{1\}$ and sp$(\cL_{is}) \subset \{|z|<1\}$.  

We need to verify that  $(\log \lambda)''(\theta)>0$ and sp$(\cL_{\theta+is}) \subset \{|z|<\lambda(\theta)\}$ for $s \neq 0$. 

To see the former we note that $$(\log \lambda)''(\theta)=\lim_{n \to \infty} \frac{1}{n}m_{\theta}\bigg(\bigg[\sum_{k=0}^{n-1}g \circ f^k\bigg]^2 h_\theta\bigg) \geq 0,$$ and if equality holds then $g$ is a continuous coboundary (see \cite[A.12b and Lemma A.16]{DL}). Therefore, in our setting, $(\log \lambda)''(\theta)>0$ for all $\theta$. 

For the latter, we first show that sp$(\cL_{\theta+is}) \subseteq \{z \in \complex\, |\, |z| \leq \lambda(\theta)\}$, essential spectral radius of $\cL_{\theta+is}$ is at most $\lambda^{-1}_* \lambda(\theta)$, and there are no eigenvalues on $\{z \in \complex\ |\ |z|=\lambda(\theta)\}$. 
%Next, we show that the essential spectral radius of $\cL_{\theta+is}$ is at most $\lambda^{-1}_*\lambda(\theta)$. To see this first note that 
Observe that
\begin{align}\label{nPower}
\cL^{n}_{\theta+is}u(x)=\sum_{f^n(y)=x}\frac{e^{(\theta+is)g_n(y)}}{ (f^n)'(y)} u(y)
\end{align}
where $g_n=\sum_{k=0}^{n-1} g \circ f^k$.
From this it follows that
\begin{align}\label{nPowerDeri}
\frac{d}{dx}\cL^{n}_{\theta+is}u = \cL^{n}_{\theta+is} \bigg(\frac{u'}{(f^n)'}+(\theta+is)\frac{g'_n}{(f^n)'}u-\frac{(f^n)''}{[(f^n)']^2}u \bigg).
\end{align}
We note that from \cite[Remark A.3]{DL} the spectral radii of $\cL_\theta : C^1 \to C^1$ and $\cL_\theta : C^0 \to C^0$ coincide. Now, from \eqref{nPower}, 
\begin{align}\label{nPower}
\|\cL^{n}_{\theta+is}u\|_\infty \leq \|\cL^n_\theta\|_{C^0}\|u\|_\infty \leq C \lambda(\theta)^n\|u\|_\infty,
\end{align}
and from \eqref{nPowerDeri},
\begin{align*}
\Big\|\frac{d}{dx}\cL^{n}_{\theta+is}u \Big\|_\infty = \|\cL^{n}_{\theta}\|_{C^0}\bigg(\lambda^{-n}_*\|u'\|_\infty+\Big[\sqrt{\theta^2+s^2}\lambda^{-n}_*\|g'_n\|_\infty + \lambda^{-2n}_*\|(f^n)''\|_\infty\Big]\|u\|_\infty \bigg).
\end{align*}
Thus, we obtain,
\begin{align*}
\|\cL^{n}_{\theta+is}u\|_{C^1} \leq  C\lambda(\theta)^n \Big( \lambda^{-n}_*\|u\|_{C^1} + \ol{C} \|u\|_\infty \Big)
\end{align*}
where $\ol{C}$ depends only on $s$ and $\theta$.
Since the unit ball in $C^1$ is relatively compact in $C^0$, we can use \cite[Lemma 2.2]{G} to conclude that the essential spectral radius of $\cL_{\theta+is}$ is at most $\lambda^{-1}_*\lambda(\theta)$ and the spectral radius of $\cL_{\theta+is}$ is at most $\lambda(\theta)$. 

Next, we normalize the family of operators $\cQ_{\theta+is}$,
$$\overline{\cQ}_{\theta+is} v(x) = \sum_{f(y)=x}\frac{e^{(\theta+is)g(y)}h_\theta(y)}{f'(y)h_\theta\circ f (y)} v(y)$$
Then, $\overline{\cQ}_{\theta+is}=H_\theta^{-1} \circ \cL_t \circ H_\theta $ where $H_\theta$ is multiplication by the function $h_\theta$. Note that $H_\theta$ is invertible because $h_\theta >0$. Now, $\cQ_{\theta+is}$ and $\overline{\cQ}_{\theta+is}$ have the same spectrum. However, the eigenfunction corresponding to the eigenvalue $1$ of $\overline{\cQ}_\theta$ changes to the constant function ${\bf 1}$. 
%(In order to apply Theorem \ref{EdgeExp} 

Assume $e^{i\theta}$ is an eigenvalue of $\overline{\cQ}_{\theta+is}$ for $s \neq 0$. Then, there exists $u \in C^1$ with $\overline{\cQ}_{\theta+is}u(x)=e^{i\theta}u(x)$. Observe that,
%\begin{multline*}
%\cL_0|u|(x)=\sum_{f(y)=x}\frac{|u(y)|}{(f'(y)}\geq \left|\sum_{f(y)=x}\frac{e^{itg(y)}u(y)}{f'(y)} \right|\\ =|\cL_tu(x)|=|e^{i\theta}u(x)|=|u(x)|
%\end{multline*}
\begin{align*}
\overline{\cQ}_\theta|u|(x)=\sum_{f(y)=x}\frac{e^{\theta g(y)}|u(y)|h(y)}{f'(y)h\circ f (y)}\geq \bigg|\sum_{f(y)=x}\frac{e^{(\theta+is)g(y)}u(y)h(y)}{f'(y)h\circ f (y)} \bigg|=|\overline{\cQ}_{\theta+is}u(x)| = |u(x)|
\end{align*}
Also note that, $\overline{\cQ}_\theta$ is a positive operator. Hence, $\overline{\cQ}^n_\theta|u|(x) \geq |u(x)|$ for all $n$. However, $$\lim_{n \to \infty} (\overline{\cQ}^n_\theta|u|)(x) = \int |u(y)| \cdot {\bf 1} \, dm_\theta(y)$$ 
because ${\bf 1}$ is the eigenfunction corresponding to the top eigenvalue. So for all $x$, $$\int |u(y)|\, d m_\theta(y)\geq |u(x)|$$ This implies that $|u(x)|$ is constant. WLOG $|u(x)|\equiv 1$. So we can write $u(x)=e^{i\gamma(x)}$ for $\gamma \in C^1$. Then, 
$$ \overline{\cQ}_{\theta+is} u(x) = \sum_{f(y)=x}\frac{e^{\theta g(y)}h(y)}{f'(y)h\circ f (y)} e^{i(sg(y)+\gamma(y))} = e^{i (\theta+\gamma(x))}. $$
Therefore,
$$\sum_{f(y)=x}\frac{e^{\theta g(y)} h(y)}{f'(y)h\circ f (y)} e^{i(sg(y)+\gamma(y)-\gamma(f(y))-\theta)} = 1 $$
for all $x$. Since, $$ \overline{\cQ}_\theta {\bf 1} = \sum_{f(y)=x}\frac{e^{\theta g(y)}h(y)}{f'(y)h\circ f (y)}  = {\bf 1}$$ 
and $e^{i(sg(y)+\gamma(y)-\gamma(x)-\theta)}$ are unit vectors, it follows that 
\begin{equation*}
sg(y)+\gamma(y)-\gamma(f(y))-\theta = 0\mod 2\pi
\end{equation*}
for all $y$. Because LHS is continuous,
\begin{equation*}
sg(y)+\gamma(y)-\gamma(f(y))-\theta = c
\end{equation*}
Because $g$ is not a continuous coboundary we have a contradiction. Therefore, $\overline{\cQ}_{\theta+is}$ does not have an eigenvalue on the unit circle when $s\neq 0$. So $\cL_{\theta+is}$ does not have eigenvalues on $\{z \in \complex\, |\, |z|=\lambda(\theta)\}$ when $s \neq 0$. 

Now, due to \Cref{FirstTerm} the strong large deviation result \eqref{Expanding} holds with $$I(a)=\sup_{\theta \in \reals}\ [a\theta -\log \lambda(\theta)]=a\theta_a - \log \lambda(\theta_a),$$ and $$K=K(a)=\frac{\sqrt{I''(a)}}{\theta_a}m_{\theta_a}(\mathbb{T})\int h_{\theta_a}(x)\rho(x)\, dx.$$ The entropy formulation of $I(a)$, \eqref{ExpandRate}, can be found in \cite[Lemma 6.6]{DL}.
\end{proof}

\begin{appendix}
\section{Construction of $\{f_k\}$.}\label{fk}
For each $k$, let $f_k(x)=\frac{1}{\pi}\tan^{-1}(kx)+\frac{1}{2}$ for $x\in[-1,k]$. Extend $f_k$ to $[-2,k+1]$ in such a way that $f_k(-2)=f_k(k+1)=0$, $f_k$ is continuously differentiable and satisfying the following conditions. 
\begin{enumerate}\setlength\itemsep{4pt}
\item $f_k$ is increasing on $[-2,k]$ with derivative on $[-2,-1]$ is bounded above by $1$. 
\item $f_k$ is decreasing on $[k+1/2,k+1]$ with derivative bounded below by $-5$. 
\item $|f'_k| \leq 5$ on $[k,k+1]$.
\item $0 \leq f_k \leq 1$ on $[-2,k+1]$ and $f_k=0$ elsewhere. 
\end{enumerate}
Then, $f_k$ is supported on $[-2,k+1]$. Here our choice of bounds $1$ and $-5$ in some sense arbitrary. As long as they are large enough and independent of $k$, we obtain an appropriate sequence of functions.

As an example, when $k=5$, the graph of $f_5$ looks like:

%\begin{center}
%\begin{tikzpicture}
%\draw [help lines, <->] (-2,0) -- (7ok,0);
%\draw [help lines, <->] (0,0) -- (0,1.1);
%\draw [green,domain=0:2*pi] plot (\x, {(sin(\x r)* ln(\x+1))/2});
%\end{tikzpicture}
%\end{center}
\begin{center}
\includegraphics[scale=0.6]{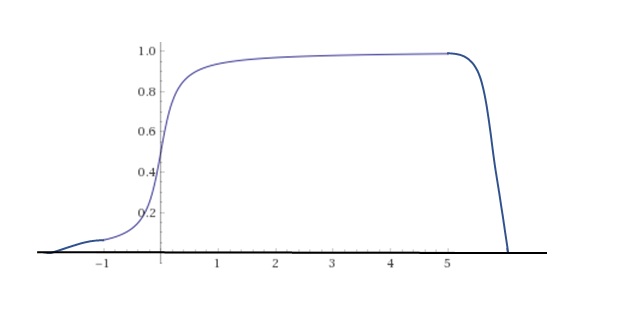}
\end{center}\vspace{-15pt}

Since $0 \leq f_k \leq 1$, for all $\gamma>0$,
\begin{align*}
\int |(f_k)_{\gamma}(x)|\, dx =\int |e^{-\gamma x}f_k(x)|\, dx \leq \int_{-2}^{\infty} e^{-\gamma x}\, dx = C_{\gamma,1} < \infty.
\end{align*}

Since $|f'_k|\leq 5$ on $[k,k+1]$, $0 \leq f_k \leq 1$ and $f_k$ is increasing on $[-2,k]$
\begin{align*}
\int |((f_k)_{\gamma})'(x)| dx &=\int_{-2}^{k+1} |\gamma e^{-\gamma x}f_k(x)+e^{-\gamma x}f'_k(x) |\ dx \\ &\leq \int_{-2}^{k+1} \Big(\gamma e^{-\gamma x}f_k(x)+e^{-\gamma x}|f'_k(x)|\Big)\, dx \\ &\leq \int_{-2}^{k} \gamma e^{-\gamma x}\, dx+\int_{-1}^{k} f'_k(x)\, dx+ \int_{k}^{k+1} ( \gamma e^{-\gamma x}+ 5e^{-\gamma x})\, dx \\ 
&\leq 1+\int_{-2}^{k+1}(5+ \gamma)e^{-\gamma x}\, dx = C_{\gamma,2} < \infty.
\end{align*}

Also, note that $|x^lf_k(x)| \leq x^le^{-\gamma x} $ for all $x \in [-2,k+1]$. Hence, 
\begin{align*}
\int |x^lf_k(x)|\, dx &\leq \int_{-2}^\infty x^le^{-\gamma x} \, dx = J_{\gamma,l} < \infty.
\end{align*}
Put $J_{r}(\gamma)=\max_{1 \leq l \leq r} J_{\gamma,l}$ and $C_{\gamma}(r) = \max \{J_r(\gamma), C_{\gamma,1}, C_{\gamma, 2}\}$. Then, $C_{\gamma}(r)$ is finite and depends only on $\gamma$ and $r$. 

Now, we have the following:%\vspace{-5pt}
\begin{enumerate}\setlength\itemsep{4pt}
\item $C^1_{r+1}((f_k)_{\gamma}) \leq C_{\gamma}(r)$ for all $k$. 
\item Since $\frac{1}{\pi}\tan^{-1}(kx)+\frac{1}{2}$ converges pointwise to $1_{[0,\infty)}(x)$, %it is easy to see that 
$f_k \to 1_{[0,\infty)}$ pointwise.
\item  Since for all $m$, $e^{-\gamma z} P^a_{m}(z)f_k(z) \to e^{-\gamma z} P^a_{m}(z)1_{[0,\infty)}(z)$ pointwise as $k \to \infty$, $$|e^{-\gamma z} P^a_{m}(z)f_k(z)| \leq e^{-\gamma z} |P^a_{m}(z)| 1_{[-2,\infty)}$$ for all $k$, and $e^{-\gamma z} |P^a_{m}(z)|1_{[-2,\infty)}$ is integrable, %we can apply the LDCT to conclude 
applying the LDCT,
$$\int P_{p}(z)
(f_k)_\gamma\left(z\right) dz = \int_{-2}^\infty e^{-\gamma z} P_{p}(z)f_k(z) \, dz \to \int_{0}^\infty e^{-\gamma z} P_{p}(z) \, dz.$$
\end{enumerate}

\subsection*{Acknowledgement} The authors would like to thank Dmitry Dolgopyat and Leonid Koralov for useful discussions and suggestions during the project and carefully reading the manuscript. While working on this article, P. Hebbar was partially supported by the ARO grant W911NF1710419.

\end{appendix}

\end{document}